\documentclass[onefignum,onetabnum]{siamart190516}
\PassOptionsToPackage{square,numbers}{natbib}
\usepackage[utf8]{inputenc} 
\usepackage[T1]{fontenc}    
\usepackage{url}            
\usepackage{booktabs}       
\usepackage{amsfonts}       
\usepackage{nicefrac}       
\usepackage{microtype}      
\usepackage{mathtools}
\usepackage{mathrsfs}
\usepackage{tikz}
\usetikzlibrary{calc,shapes,arrows}

\usepackage{lipsum}
\usepackage{amsfonts}
\usepackage{graphicx}
\usepackage{epstopdf}
\usepackage{algorithmic}
\ifpdf
  \DeclareGraphicsExtensions{.eps,.pdf,.png,.jpg}
\else
  \DeclareGraphicsExtensions{.eps}
\fi

\newsiamremark{remark}{Remark}
\newsiamremark{hypothesis}{Hypothesis}
\crefname{hypothesis}{Hypothesis}{Hypotheses}
\newsiamthm{claim}{Claim}
\newsiamthm{assumption}{Assumption}
\headers{Strong Second-Order Theory for Decomposable Functions}{W. Ouyang and A. Milzarek}

\title{Variational Properties of Decomposable Functions. Part II: Strong Second-Order Theory\thanks{Submitted to the editors DATE.
\funding{Andre Milzarek was partly supported by the National Natural Science Foundation of China (Foreign Young Scholar Research Fund Project) under Grant No$.$ 12150410304, by the Shenzhen Science and Technology Program under Grant No$.$ RCYX20221008093033010, and by Shenzhen Stability Science Program 2023, Shenzhen Key Lab of Multi-Modal Cognitive Computing.}}}
\author{Wenqing Ouyang\thanks{School of Data Science, The Chinese University of Hong Kong, Shenzhen (CUHK-Shenzhen), Guangdong, 518172, P.R. China (\email{wenqingouyang1@link.cuhk.edu.cn} and \email{andremilzarek@cuhk.edu.cn}).}
  \and Andre Milzarek\footnotemark[2]
 }

\usepackage{amsopn}
\DeclareMathOperator{\diag}{diag}

\usepackage{relsize}
\usepackage{scalerel}
\usepackage{comment}
\tikzset{fontscale/.style = {font=\relsize{#1}}
    }
\usepackage{pgfplots}
\usepgfplotslibrary{fillbetween}

\usepackage{color}
\usepackage{graphicx}
\graphicspath{{Figs/}}
\usepackage{indentfirst,latexsym,bm} 
\usepackage{longtable}
\usepackage{cuted}
\usepackage{booktabs,multirow,array,multicol}
\usepackage{enumitem}

\hypersetup{hypertex=true,
            colorlinks=true,
            linkcolor=blue,
            anchorcolor=blue,
            citecolor=blue}

\DeclareMathOperator*{\argmin}{argmin}
 \newcommand\conv{\mathrm{conv}}
\usepackage{subfigure}
\usepackage{amssymb}
\usepackage{mdframed}
\definecolor{ivory}{RGB}{208,205,193}
\definecolor{ivoryt}{RGB}{198,195,183}

\definecolor{cuhkp}{RGB}{98,56,105} 	
\definecolor{cuhkpl}{RGB}{152,24,147} 	
\definecolor{cuhkb}{RGB}{219,160,1} 	
\definecolor{cuhkbd}{RGB}{178,129,0} 	
\definecolor{cuhkr}{RGB}{88,35,155}  	

\definecolor{oranl}{RGB}{223,149,86}
\definecolor{turql}{RGB}{53,130,134}  		

\definecolor{blackp}{RGB}{0,0,0} 
\definecolor{redp}{RGB}{255,0,0}
\definecolor{orangep}{RGB}{255,128,0}
\definecolor{brownp}{RGB}{128,77,0}
\definecolor{yellowp}{RGB}{255,230,0}
\definecolor{greenp}{RGB}{128,230,0}
\definecolor{bluep}{RGB}{0,128,255}
\definecolor{purplep}{RGB}{152,24,147}
\definecolor{pinkp}{RGB}{230,0,128}       
\definecolor{amg}{RGB}{17,140,17}

\mdfdefinestyle{codebox}{
backgroundcolor=gray!3,
    linecolor=black, 
    linewidth=0.7pt,
    topline=true,
    bottomline=true,
    rightline=true,
    leftmargin=0pt,
    innertopmargin=8pt,
    innerbottommargin=8pt,
    innerrightmargin=10pt,
    innerleftmargin=10pt,
    shadow=false,
    shadowcolor=gray!10,
    shadowsize=2pt,
    roundcorner=0pt,
}

\mdfdefinestyle{highlighta}{
backgroundcolor=white,
    linecolor=black, 
    linewidth=0.5pt,
    topline=true,
    bottomline=true,
    rightline=true,
    leftmargin=0pt,
    innertopmargin=-3pt,
    innerbottommargin=6pt,
    innerrightmargin=8pt,
    innerleftmargin=8pt,
    shadow=false,
    shadowcolor=gray!10,
    shadowsize=2pt,
    roundcorner=0pt,
    skipabove=0pt,
    skipbelow=0pt,
}

\mdfdefinestyle{highlightb}{
backgroundcolor=white,
    linecolor=black, 
    linewidth=0.5pt,
    topline=true,
    bottomline=true,
    rightline=true,
    leftmargin=0pt,
    innertopmargin=-6pt,
    innerbottommargin=6pt,
    innerrightmargin=8pt,
    innerleftmargin=8pt,
    shadow=false,
    shadowcolor=gray!10,
    shadowsize=2pt,
    roundcorner=0pt,
    skipabove=0pt,
    skipbelow=0pt,
}
\newtheorem{thm}{theorem}[section]

\newtheorem{lem}[thm]{Lemma}
\newtheorem{prop}[thm]{Proposition}
\newtheorem{rem}[thm]{Remark}
\newcommand{\revise}[1]{{\color{black}#1}}
\newcommand{\myrevise}[1]{{\color{black}#1}}
\newtheorem{defn}[thm]{Definition}
\newtheorem{assum}[thm]{Assumption}

\newtheorem{example}[thm]{Example}
\usepackage[capitalize,nameinlink]{cleveref}

\newcommand{\RNum}[1]{\uppercase\expandafter{\romannumeral #1\relax}}
\newcommand{\rnum}[1]{\expandafter{\romannumeral #1\relax}}

\newcommand\Fnors{F^{\tau}_{\mathrm{nor}}}
\newcommand\cM{\mathcal{M}^{\tau}}

\newcommand\cC{\mathcal{C}}

\newcommand\cR{\mathcal{R}}
\newcommand\cS{\mathcal{S}}
\newcommand\epii{\mathrm{epi}}

\newcommand\cN{\mathcal{N}}

\newcommand\cG{\mathcal{G}}
\newcommand\bS{\mathbb{S}}
\newcommand{\D}{\mathcal D}
\newcommand\ri{\mathrm{ri}}
\newcommand\dist{\mathrm{dist}}
\newcommand\dom{\mathrm{dom}}
\newcommand\lin{\mathrm{lin}}
\newcommand\spa{\mathrm{span}}
\newcommand{\nlp}{\mathrm{(NLP)}}
\newcommand\rd{\mathrm{d}}

\DeclareMathOperator*{\elim}{e-lim}
\DeclareMathOperator*{\eliminf}{e-liminf}
\newcommand\aff{\mathrm{aff}}

\newcommand\proxs{\mathrm{prox}_{\tau\varphi}}

\newcommand\R{\mathbb{R}}
\newcommand\Rn{\mathbb{R}^n}
\newcommand\Rm{\mathbb{R}^m}
\newcommand\n{\mathbb{N}}
\newcommand\Rex{(-\infty,\infty]}
\newcommand\Rexx{[-\infty,\infty]}

\newcommand{\iprod}[2]{\langle #1, #2 \rangle}
\newcommand\vp{\varphi}
\newcommand\half{\frac12}

\newcommand\Fnor[1]{F^{\tau}_{\mathrm{nor}}(#1)}
\newcommand\oFnor{F^{\tau}_{\mathrm{nor}}}
\newcommand\Fnat[1]{F^{\tau}_{\mathrm{nat}}(#1)}
\newcommand\oFnat{F^{\tau}_{\mathrm{nat}}}

\newcommand\gph{\mathrm{gph}}

\newcommand\crit{\mathrm{crit}}
\newcommand\SSOSC{\mathrm{SSOSC}}
\newcommand{\affS}{\mathsf{A}}

\newcommand{\be}{\begin{equation}}
\newcommand{\ee}{\end{equation}}
\newcommand{\proj}{\mathrm{proj}_X}

\makeatletter

\newcommand{\Rmnum}[1]{\expandafter\@slowromancap\romannumeral #1@}
\makeatother

\newcommand\pvpd{\mathcal{Q}}
\newcommand\usotp{\mathcal{UTP}}
\newcommand\sotp{\mathcal{OTP}}

\crefname{section}{section}{sections}
\crefname{subsection}{subsection}{subsections}

\Crefname{figure}{Figure}{Figures}

\crefformat{equation}{\textup{#2(#1)#3}}
\crefrangeformat{equation}{\textup{#3(#1)#4--#5(#2)#6}}
\crefmultiformat{equation}{\textup{#2(#1)#3}}{ and \textup{#2(#1)#3}}
{, \textup{#2(#1)#3}}{, and \textup{#2(#1)#3}}
\crefrangemultiformat{equation}{\textup{#3(#1)#4--#5(#2)#6}}%
{ and \textup{#3(#1)#4--#5(#2)#6}}{, \textup{#3(#1)#4--#5(#2)#6}}{, and \textup{#3(#1)#4--#5(#2)#6}}

\Crefformat{equation}{#2Equation~\textup{(#1)}#3}
\Crefrangeformat{equation}{Equations~\textup{#3(#1)#4--#5(#2)#6}}
\Crefmultiformat{equation}{Equations~\textup{#2(#1)#3}}{ and \textup{#2(#1)#3}}
{, \textup{#2(#1)#3}}{, and \textup{#2(#1)#3}}
\Crefrangemultiformat{equation}{Equations~\textup{#3(#1)#4--#5(#2)#6}}%
{ and \textup{#3(#1)#4--#5(#2)#6}}{, \textup{#3(#1)#4--#5(#2)#6}}{, and \textup{#3(#1)#4--#5(#2)#6}}

\crefdefaultlabelformat{#2\textup{#1}#3}

\hypersetup{
  pdftitle={Variational Properties of Decomposable Functions. Part II: Strong Second-Order Theory},
  pdfauthor={W. Ouyang and A. Milzarek}
}

\begin{document}

\maketitle

\begin{abstract}
  Local superlinear convergence of the semismooth Newton method usually necessitates assumptions on the uniform invertibility of the utilized, generalized Jacobian matrices, such as, e.g., BD- or CD-regularity. For certain composite-type problems and nonlinear programs (for which explicit representations of the generalized Jacobians of the associated stationarity equations are available), such regularity assumptions are closely connected to strong second-order sufficient conditions. However, general characterizations are not well understood. In this paper, we investigate a strong second-order sufficient condition ($\SSOSC$) for composite problems whose nonsmooth part has a generalized conic-quadratic second subderivative. We  discuss the relationship between the $\SSOSC$ and other second order-type conditions that involve the generalized Jacobians of the normal map. In particular, these two conditions are equivalent under certain structural assumptions on the generalized Jacobian matrix of the proximity operator. Leveraging second-order variational techniques and properties, we then verify that the introduced structural conditions hold for a broad class of $C^2$-strictly decomposable functions. Finally, it is shown that the $\SSOSC$ is further equivalent to the strong metric regularity of the subdifferential, the normal map, and the natural residual. 
\end{abstract}




\section{Introduction}
In this work, we consider composite-type, nonconvex, nonsmooth optimization problems of the form:
\begin{align}
    \label{prob1}
    \min_{x\in\R^n}~\psi(x):=f(x)+\vp(x),
\end{align}
where $f : \Rn \to \R$ is twice continuously differentiable and $\vp : \Rn \to \Rex$ is a lower semicontinuous (lsc), proper, and prox-bounded mapping. We are interested in second-order variational properties of the problem \cref{prob1} and their potential implications for second-order-type methodologies. 

In this context, a common strategy is to express the associated optimality conditions of \cref{prob1} as a nonsmooth equation and to utilize the semismooth Newton method, \cite{qi1993nonsmooth}, to solve this equation. 
 For ease of exposition, let us suppose for a moment that $\vp$ is weakly convex and $\tau$ is sufficiently small such that the proximity operator $\proxs$ is single-valued. The natural residual equation,
\begin{align}
  \label{nat_res}
  \oFnat(x)=0, \quad \oFnat(x):=x-\proxs(x-\tau\nabla f(x)).
\end{align}
defines a popular representation of the first-order optimality conditions of \cref{prob1}, \cite{fukushima1981proximal}. 
Another possible choice is the normal map equation, \cite{robinson1992normal,facchinei2003finite,pieper2015finite}, which is given by:
\begin{align}
  \label{nor_map}
  \oFnor(z)=0, \quad \oFnor(z):=\nabla f(\proxs(z))+{\tau^{-1}}(z-\proxs(z)).
\end{align} 
Taking the normal map as an example, in each iteration, the semismooth Newton method generates a new step via:
\begin{align}
  \label{nor_ssn_iter}
  z^{k+1}=z^k-M_k^{-1}\Fnors(z^k), \quad M_k \in \cM_{\mathrm{nor}}(z^k),
\end{align}
where $\cM_{\mathrm{nor}}(z):=\{ \nabla^2f(\proxs(z))D+\frac{1}{\tau}(I-D): D\in \partial\proxs(z)\}$ serves as a set of generalized Jacobian matrices of $\oFnor$, see, e.g., \cite{pieper2015finite,kunisch2016time,mannel2021hybrid,ouyang2021trust} and \cite{milzarek2014semismooth,ByrChiNocOzt16,milzarek2016numerical,stella2017forward} for counterparts using the natural residual $\oFnat$. The superlinear convergence of the semismooth Newton steps typically requires the uniform invertibility of the matrices $M_k$ around the solution, i.e., the norm of the inverse of $M_k$ has to remain bounded by some $C$ for all sufficiently large $k$. This requirement is also related to the CD-regularity of the normal map $\oFnor$. 
One of the main goals of this paper is to address the question:\vspace{1.2ex}

\begin{mdframed}[style=codebox]
\emph{Can the invertibility condition, $\|M_k^{-1}\|\leq C$, be ensured by suitable second-order (sufficient) conditions for \cref{prob1} (similar to the smooth case $\vp \equiv 0$)?}
\end{mdframed}
\subsection{Related Works}
The semismooth Newton method was initially proposed by Qi and Sun for solving nonsmooth equations, \cite{qi1993nonsmooth}. Due to its numerical efficiency, it has been widely applied in various fields,  \cite{facchinei2003finite,ulbrich2011semismooth,zhao2010newtoncg,ByrChiNocOzt16,li2018efficient,milzarek2019stochastic,lin2019clustered}. The BD- and CD-regularity play central roles to guarantee the fast local superlinear convergence of the semismooth Newton method \cite{kummer1988newton,kummer1992newton,qi1993convergence,qi1993nonsmooth,facchinei2003finite}. The theoretical justification of such regularity assumptions is also a key research direction in the analysis of nonsmooth Newton methods. Robinson first introduced the concept of strong regularity of the solution to a generalized equation and the strong second-order sufficient condition in \cite{robinson1980strongly}, which includes nonlinear programming as a special case. Moreover, Robinson showed that the strong second-order sufficient condition together with the LICQ (linear independence constraint qualification) implies the strong regularity of the solution and the converse is also true for smooth nonlinear programs, 
see \cite[Proposition 5.38]{bonnans2013perturbation} and \cite[Theorem 5]{dontchev1996characterizations}. 
In \cite{bonnans2005perturbation}, under Robinson's constraint qualification, Bonnans and Ramírez verify that for nonlinear second-order cone programming, at a local minimizer, the strong regularity of the KKT solution mapping is equivalent to the strong second-order sufficient condition. In \cite{sun2006strong,chan2008constraint}, it is further shown that these notions are equivalent to the nonsingularity of Clarke's Jacobian matrix of the KKT system for semidefinite programming. This links the CD-regularity of the stationarity equation to the strong second-order sufficient condition. Similar results are available for specific problems whose stationarity systems have explicit generalized Jacobians, which \revise{include}, e.g., Ky-Fan $k$-norm-based problems \cite{ding2014introduction}, matrix spectral conic programming~\cite{guo2015some}, $\ell_1$-conic programming \cite{liu2019characterizations}, and second-order cone programming (SOCP) \cite{wang2010nonsingularity}. Recently, a notion called variational sufficiency, \cite{rockafellar2023augmented}, is shown to be equivalent to strong regularity when the set of constraints is polyhedral \cite{rockafellar2023augmented} or related to the cone of positive semidefinite matrices \cite{wangstrong}. The framework proposed in \cite{rockafellar2023augmented} is applicable in the composite optimization setting, but direct and novel connections between the variational sufficiency condition and the uniform invertibility of the matrices $M_k$, $k\in\mathbb N$, in \cref{nor_ssn_iter} seem mostly unavailable, due to the intrinsic hardness of calculating the generalized Jacobian of the stationarity equation. Notably, for the subspace containing derivative (SCD) mapping, BD-regularity was proven to be equivalent to the so-called SCD regularity, \cite{gfrerer2022scd}, which allows establishing the local superlinear convergence of the semismooth Newton method. The strong second-order sufficient condition is also used in \cite{wang2009properties,wangstrong,kato2007sqp} to prove superlinear convergence of certain second-order-type methods.  

In the context of composite optimization, it is shown in \cite{stella2017forward,ouyang2021trust} that when the stationarity measure is locally smooth at the solution (which typically necessitates the strict complementarity condition to hold), the basic second-order sufficient condition can imply the uniform invertibility of the matrices $M_k$, $k \in \mathbb N$ around the solution. A similar result is established in \cite[Lemma 5.4.4]{milzarek2016numerical} by directly assuming the strict complementarity condition and $C^2$-fully decomposability of $\vp$. 

Without the strict complementarity condition and such local smoothness assumption, there is much less one can say about the characterization of the uniform invertibility of the generalized Jacobian of the natural residual or normal map. In \cite{hale2008fixed,stadler2009elliptic,wen2012convergence,milzarek2014semismooth}, $\ell_1$-optimization problems are considered where $\vp$ is the $\ell_1$-norm. It is shown that the strong variant of the second-order sufficient condition ensures the uniform invertibility of the generalized Jacobian of the natural residual. However, it is unknown if similar results hold for a broader and more general class of functions.


It is well-known that strong metric regularity is equivalent to metric regularity for monotone operators, \cite{kenderov1975semi}. For convex functions, strong metric regularity is further equivalent to the uniform quadratic growth condition, \cite{artacho2008characterization}. For general lsc functions, strong metric regularity of the subdifferential is also equivalent to the local uniform growth condition at a local minimizer, \cite{DruMorNhg14}. Moreover, these conditions are equivalent to the tilt stability of the local minimizer provided that the function is prox-regular and subdifferentially continuous. Tilt stability can be also characterized via the positive definiteness of the generalized Hessian, see \cite{DruMorNhg14}. We are not aware of any literature concerning the equivalence between strong metric regularity of the subdifferential and strong metric regularity of the natural residual, let alone the normal map.


\subsection{Contributions}
We investigate a general strong second-order sufficient condition for problem \cref{prob1} and its implications for uniform invertibility assumptions and the semismooth Newton method. Since the second-order differential utilized by Mordukhovich et al$.$ in \cite{mordukhovich2015full} lacks direct connections to the proximity operator, our analysis and formulation of the strong second-order sufficient condition will be based on the second subderivative of $\psi$, which 
is closely related to the directional derivative of the proximity operator and its generalized Jacobian, \cite{poliquin1996generalized,rockafellar2009variational}. 
The standard second-order sufficient conditions for \cref{prob1} are given by 
\[ \rd^2\psi(\bar x|0)(h)\geq \sigma\|h\|^2 \quad \forall~h \in \cC(\bar x), \] 
where $\sigma > 0$ and $\cC(\bar x):=\{h\in\R^n:\rd\psi(\bar x)(h)=0\}$ is the critical cone, \cite{rockafellar2009variational}. Mimicking the corresponding counterparts in nonlinear programming, the strong second-order sufficient condition for \cref{prob1} should be expressed as a quadratic growth condition of some proper extension of $\rd^2\psi(\bar x|0)$ on the affine hull of $\cC(\bar x)$. As this is not straight- forward for general $\psi$, 
we narrow our focus to special nonsmooth functions, whose second subderivatives have a generalized conic quadratic structure. It turns out that the class of $C^2$-strictly decomposable is highly ``amenable'' for our purposes, i.e., the second subderivative of such functions is generalized conic quadratic and, in addition, a finer characterization of the generalized Jacobian matrix of the proximity operator is available. Here, $\vp$ is $C^2$-strictly decomposable at a stationary point $\bar x$, if $\vp=\sigma_{\pvpd}\circ F$ is the composition of a sublinear mapping $\sigma_{\pvpd}$ and a smooth mapping $F$, a suitable constraint qualification holds at $\bar x$, and we have $F(\bar x)=0$. 
These additional properties allow us to link the strong second-order sufficient condition to the uniform invertibility of the matrices $M_k$ used in \cref{nor_ssn_iter}. Surprisingly, strong metric regularity of the normal map, the subdifferential $\partial\psi$, and the natural residual can also be shown to be equivalent to the strong second-order sufficient condition. Our new results provide theoretical guarantees for the uniform invertibility assumption in the convergence analysis of the semismooth Newton method without requiring the ubiquitous strict complementarity condition, \cite{milzarek2016numerical,stella2017forward,ouyang2021trust}. In summary, our contributions include:
\begin{itemize}
  \item We propose a strong second-order sufficient condition for nonsmooth functions whose second subderivative has a generalized conic quadratic structure, which includes the class of $C^2$-strictly decomposable functions as a special case. 
  \item We prove that the proposed strong second-order sufficient condition is equivalent to a second-order condition that is based on the generalized Jacobian of the proximity operator. This condition can be shown to be equivalent to the uniform invertibility of the matrices $M_k$ in \cref{nor_ssn_iter} under certain algebraic assumptions on $\partial\proxs$. Leveraging variational techniques, we verify that such algebraic conditions hold for $C^2$-strictly decomposable functions when a mild geometric assumption on the support set $\pvpd$, ($\vp=\sigma_{\pvpd}\circ F$), is imposed. 
  \item Finally, we show that the proposed strong second-order sufficient condition is further equivalent to the strong metric regularity of the subdifferential $\partial\psi$, the natural residual $\oFnat$, and the normal map $\oFnor$. 
\end{itemize}

\subsection{Notation and Variational Tools} \label{sec:notation}
\textit{General Notation}. Our terminologies and notations follow the standard textbooks \cite{rockafellar1970convex,bonnans2013perturbation,rockafellar2009variational}. 
By $\iprod{\cdot}{\cdot}$ and $\|\cdot\|$, we denote the standard Euclidean inner product and norm. For matrices, the norm $\|\cdot \|$ is the standard spectral norm. The sets of symmetric and symmetric, positive semidefinite $n \times n$ matrices are denoted by $\mathbb S^n$ and $\mathbb{S}_{+}^n$, respectively. For two matrices $A, B \in \mathbb S^n$, we write $A\succeq B$ if $A-B$ is positive semidefinite. 

\textit{Sets}. For a sequence of sets $\{C_k\}_k \subseteq \R^n$, the outer and inner limit are defined as:
\begin{align*}
  &{\limsup}_{k\to\infty}\,C_k:=\{x\in\R^n: \exists~\n\ni j_k\to\infty,~x_{j_k}\in C_{j_k},~x_{j_k}\to x\},\\
  &{\liminf}_{k\to\infty}\,C_k:=\{x\in\R^n: \exists~x_{k}\in C_{k},~x_{k}\to x\}.   
\end{align*}
Following the standard Painlev\'{e}-Kuratowski convergence for sets, we write $C_k\to C$ when both limits coincide. This can also be equivalently characterized by the point-wise convergence of the distance functions $\{d_{C_k}\}_k$ to $d_C$. For any closed set $C \subseteq \Rn$, $\Pi_C : \Rn \rightrightarrows \Rn$ denotes the projection mapping onto $C$. If $C$ is closed and convex, then $\Pi_C$ is a single-valued function. In addition, if $C$ is a linear subspace, we will identify $\Pi_C$ with the corresponding projection matrix. We use $\iota_C$ and $\sigma_C$ to denote the indicator and support function of set $C$, respectively. The set $\aff(C)$ is the affine hull of $C$, $\lin(C)$ denotes the lineality space of $C$, and the relative interior of a convex set $C \subseteq \Rn$ is given by $\mathrm{ri}(C)$. 

We also require the notion of second-order tangent sets. \cref{def:tangent-sets} is taken from \cite[Definition 2.54 and 3.28]{bonnans2013perturbation}.
\begin{defn} \label{def:tangent-sets}
    Assume that $C \subseteq \Rn$ is locally closed at $x\in C$. The tangent cone $T_C(x)$ of $C$ at $x$ is given by $T_C(x):=\limsup_{t\downarrow 0}{(C-x)}/{t}$.
    The outer second-order tangent set of $C$ at $x$ for $h$ is defined as
    \[   T_C^2(x,h):=\limsup_{t\downarrow0}\frac{C-x-th}{\frac12 {t^2}}=\left\{w:\exists~t_k\downarrow 0, \,\dist(x+t_kh+\tfrac{t_k^2}{2}w,C)=o(t^2_k)\right\}  \]
    and the inner second-order tangent set of $C$ at $x$ for $h$ is given by:
    \[   T_C^{i,2}(x,h):=\liminf_{t\downarrow0}\frac{C-x-th}{\frac12{t^2}}=\left\{w: \dist(x+th+\tfrac{t^2}{2}w,C)=o(t^2), \, t\downarrow 0\right\}.  \]
\end{defn}

For a convex set $C \subseteq \Rn$, the normal cone to $C$ at $x \in C$ is given by $N_C(x) = T_C(x)^\circ = \{v: \iprod{v}{y-x} \leq 0, \, \forall~y \in C\}$ where $\cdot^{\circ}$ denotes the polar cone operation.


\textit{Functions, epi-convergence, derivatives}. The effective domain of an extended real-valued function $\theta : \Rn \to \Rexx$ is given by $\dom{(\theta)}=\{x \in \Rn : \theta(x)<\infty\}$ and $\theta$ is called proper if $\theta(x) > -\infty$ and $\dom{(\theta)} \neq \emptyset$. The mapping $\theta$ is said to be $\rho$-weakly convex, $\rho \geq 0$, if $\theta+\frac{\rho}{2}\|\cdot\|^2$ is convex. The set $\epii{(\theta)} :=\{(x,t)\in\R^n\times \R:  \theta(x) \leq t\}$ denotes the epigraph of $\theta :\R^n\to\Rex$. A sequence of extended real-valued functions $\{\theta_k\}_k$ epi-converges to $\theta$ if $\epii(\theta_k) \to\epii(\theta)$. 
Epi-convergence of $\{\theta_k\}_k$ to $\theta$ is equivalent to the condition
\be \label{eq:def-epi} \left[ \begin{array}{ll} \liminf_{k \to \infty}~\theta_k(x^k) \,\, \geq \theta(x) & \text{for every sequence } x^k \to x, \\[.5ex] \limsup_{k \to \infty}~\theta_k(x^k) \leq \theta(x) & \text{for some sequence } x^k \to x, \end{array} \right. \quad \forall~x. \ee
If the property \cref{eq:def-epi} holds for fixed $x \in \Rn$, then the epi-limit of $\{\theta_k\}_k$ at $x$ exists and we write $(\elim_{k \to \infty} \theta_k)(x) = \theta(x)$. The  lower epi-limit is given by $(\eliminf_{k \to \infty} \theta_k)(x) = \liminf_{k \to \infty, \, \tilde x \to x} \theta_k(\tilde x)$, cf. \cite[Chapter 7]{rockafellar2009variational}. 

For $\theta:\R^n\to (-\infty,\infty]$ and $x \in\dom(\theta)$, the regular subdifferential is defined as $\hat{\partial}\theta(x):=\{v\in \R^n: \theta(y)\geq \theta(x)+\langle  v,y-x\rangle+o(\|y-x\|)\}$. The subdifferential of $\theta$ at $x$ is then given by $\partial \theta(x):=\limsup_{y \to x, \, \theta(y) \to \theta(x)}\hat{\partial} \theta(y)$. The second subderivative of $\theta$ at $x \in \dom(\theta)$ relative to $\alpha$ in the direction $h \in \Rn$ is the lower epi-limit
%
%
%
%
\[ \mathrm{d}^2\theta(x|\alpha)(h) = \liminf_{t\downarrow 0, \, \tilde h \to h}~\Delta_t^2  \;\! \theta(x|\alpha)(\tilde h), \quad \Delta_t^2  \;\! \theta(x|\alpha)(h) := \frac{\theta(x+th)-\theta(x)-t \iprod{\alpha}{h}}{\half t^2}. \]
%
We say that $\theta$ is {twice epi-differentiable} at $x$ for $\alpha \in \Rn$ if the second-order difference quotients $ \Delta_t^2  \;\! \theta(x|\alpha)$ epi-converge in the sense of \cref{eq:def-epi}. 

For a mapping $F:\mathbb{R}^{n}\rightarrow\mathbb{R}^m$, the set $\mathcal{R}(F):=\{y \in \R^m: \exists~x \in \Rn \, \text{with} \, y=F(x)\}$ is the range of $F$. \revise{For a matrix $Q \in \R^{m \times n}$, $\cR(Q)$ is used to denote the range of $Q$.}
%
For a locally Lipschitz continuous, vector-valued mapping $F:\R^n\to \R^m$, the set $\partial_B F(x)$ denotes the Bouligand subdifferential of $F$ at $x$, i.e.,
\[ \partial_B F(x):=\{V\in \R^{m\times n}: x^k\to x,~\text{$F$ is differentiable at $x^k$},~ \revise{\D F(x^k)}\to V\}.                          \]
Clarke's subdifferential (generalized Jacobian) is defined as $\partial F(x) =\conv(\partial_B F(x))$. For $F:\R^n\rightrightarrows \R^m$, $\gph(F) := \{(x,y) \in \Rn \times \Rm: y \in F(x)\}$ is the graph of $F$. Finally, we introduce the notion of metric regularity for set-valued mappings, cf. \cite{DonRoc14}. 
\begin{defn}
    A set-valued mapping $F:\R^n\rightrightarrows\R^m$ is metrically regular at $\bar x$ for $\bar y\in F(\bar x)$, if $\gph(F)$ is locally closed at $(\bar x,\bar y)$ and there are $\kappa \geq 0$ and a neighborhood $U \times V$ of $(\bar x,\bar y)$ such that:
    \[    \dist(x,F^{-1}(y))\leq \kappa\dist(y,F(x)) \quad \forall~x\in U, \quad \forall~y\in V.          \]
  If, in addition, $F^{-1}$ has a single-valued localization around $\bar y$ for $\bar x$, then $F$ is said to be strongly metrically regular at $\bar x$ for $\bar y$.
\end{defn}

We note that a single-valued continuous mapping $F:\R^n\to\R^n$ is strongly metrically regular at $\bar x$ for $F(\bar x)$ iff there is a neighborhood $U$ of $\bar x$ and $\sigma>0$ such that $\sigma \|y-z\| \leq \|F(y)-F(z)\|$ for all $y,z\in U$, cf. \cite[Theorem 1F.2]{DonRoc14}. 


\textit{Proximal properties}. For a proper, lsc function $\vp : \Rn \to \Rex$ and $\tau > 0$, the proximity operator of $\vp$ is defined as
\[ \proxs : \Rn \rightrightarrows \Rn, \quad \proxs(x) \in \argmin_{y\in\Rn}~\vp(y) + \frac{1}{2\tau}\|x-y\|^2, \]
see \cite{rockafellar2009variational}. The parameter $\tau$ used in the definition of $\proxs$ is assumed to be positive throughout this paper. Next, we define a useful regularity concept for $\vp$. 
%
\begin{defn}
    \label{defn2-1}
   We say that $\vp$ is globally prox-regular at $\bar x$ for $\bar v\in \partial \vp(\bar x)$ with constant $\rho$, if $\vp$ is locally lsc at $\bar x$, and there is $\epsilon>0$ such that 
    \[ \forall~x'\in\R^n ,\quad \vp(x')\geq \vp(x)+\langle v,x'-x \rangle-\frac{\rho}{2}\|x'-x\|^2          \]
     for all $\|x-\bar x\|<\epsilon,\|v-\bar v\|<\epsilon,~v\in \partial\vp(x),~\vp(x)<\vp(\bar x)+\epsilon$.  
  \end{defn}

Global prox-regularity is a mild condition and is implied by the prox-boundedness and the (standard) prox-regularity, see \cite[Proposition 8.49(f) and 13.37]{rockafellar2009variational}. 

Let us reconsider problem \cref{prob1}. The set of stationarity points of \cref{prob1} is given by $\crit(\psi) := \{x \in \Rn: 0 \in \partial\psi(x)\}$. Let $\vp$ be globally prox-regular (or prox-bounded and prox-regular) at $x $ for $v = -\nabla f(x)$ with constant $\rho$. Then, for $\tau > 0$ and $\tau\rho < 1$, the proximity operator $\proxs$ is single-valued and Lipschitz continuous around $z = x+\tau v$, cf. \cite{rockafellar2009variational,ouyang2023partI}. In this case, we can define the generalized \revise{derivatives} of $\oFnat$ and $\oFnor$:
\begin{align*} \mathcal M_{\mathrm{nat}}^\tau(x) & := \{M: M = I - D(I-\tau\nabla^2 f(x)), \, D \in \partial \proxs(x-\tau\nabla f(x)) \}, \\ \mathcal M_{\mathrm{nor}}^\tau(z) & := \{M: M = \nabla^2 f(\proxs(z))D + \tfrac{1}{\tau} (I - D), \, D \in \partial \proxs(z) \}. \end{align*}
By \cite{clarke1990optimization}, we have $\partial\Fnor{z}h = \mathcal M_{\mathrm{nor}}^\tau(z)h$ for all $h \in \Rn$.
In addition, if $x \in \crit(\psi)$, then it follows $\Fnat{x} = 0$, $\Fnor{z} = 0$, and $x = \proxs(z)$, see, e.g., \cite[Lemma 2.1]{ouyang2021trust}.   

 \textit{Convex geometry}. 
Finally, we discuss several concepts from convex geometry that will be used in \Cref{sec:sosub2}. The face and exposed face of a convex set are defined as follows, cf. \cite{HirLem01,weis2010note}:
\begin{defn}
  For a convex set $C \subseteq \Rn$, a face $F$ of $C$ is a subset of $C$ with the property that $\lambda y+(1-\lambda)z=x\in F$ with $y,z\in C$ and $\lambda\in(0,1)$ implies $[y,z]\subseteq F$. An exposed face $F_{\mathrm{ex}}$ is a subset of $C$ which can be expressed as $F_{\mathrm{ex}} = H\cap C$, where $H$ is a supporting hyperplane of $C$. Equivalently, $F_{\mathrm{ex}}$ is an exposed face of $C$ if there exists $x \neq 0$ such that $F_{\mathrm{ex}} = \{y \in C : \iprod{x}{y} = \sigma_C(x)\}$. Let $\mathcal F_{\mathrm{ex}}$ denote the set of all exposed faces of $C$. The smallest exposed face of $C$ containing $x \in C$ is denoted by $G_{\mathrm{ex}}(x) := \bigcap\,\{ F \in \mathcal F_{\mathrm{ex}}: x \in F\}$.
\end{defn} 

 Noticing $\partial\sigma_C(x) = \{y \in C: \iprod{x}{y} = \sigma_C(x)\}$ \revise{(see, e.g., \cite[Example 11.4]{rockafellar2009variational})}, each exposed face $F_{\mathrm{ex}}$ of the set $C$ can be represented by $\partial\sigma_C(x)$ for an appropriate point $x \neq 0$. Next, we collect several well-known properties of the smallest exposed face $G_{\mathrm{ex}}$, cf. \cite[Lemma 4.6]{weis2010note} and \cite[Lemma 2.2.2]{schneider2014convex}. We provide a self-contained proof of \cref{lem:exposed} in \cref{app:exposed}.
\begin{lem} \label{lem:exposed} Let $C \subseteq \Rn$ be a convex set and let $x \in C$ be given. 
\begin{enumerate}[label=\textup{\textrm{(\roman*)}},topsep=0pt,itemsep=0ex,partopsep=0ex]
\item If $y \in \ri(G_{\mathrm{ex}}(x))$, then we have $N_C(x) = N_C(y)$ and $T_C(x) = T_C(y)$.
\item It holds that $G_{\mathrm{ex}}(x) = \partial\sigma_C(y)$ for all $y \in \ri(N_C(x))$.
\end{enumerate}
\end{lem} 
 

\subsection{Organization}
This paper is organized as follows. In \Cref{sec:ssonc}, we introduce the strong second-order sufficient optimality condition ($\SSOSC$) and a related second-order-type condition that is based on the generalized derivatives of the normal map $\oFnor$. We then formulate structural assumptions on $\partial\proxs$ to ensure equivalence between those two optimality conditions. In \Cref{sec:eq_ssonc}, we analyze the second-order variational properties of strictly decomposable functions and we verify that the postulated structural assumptions on $\partial\proxs$ hold if $\vp$ is $C^2$-strictly decomposable under a mild, additional assumption on the support set of the decomposition. In \Cref{sec:eq_ssosc}, we prove equivalence between the $\SSOSC$ and other variational concepts, including the strong metric regularity of the subdifferential $\partial \psi$, the normal map, and the natural residual, and several other second-order-type conditions. 


\section{Connecting the $\SSOSC$ and Generalized Jacobians}
\label{sec:ssonc}
In this section, we formally define the $\SSOSC$ and discuss its equivalence to a second-order condition which can ensure the uniform invertibility of the matrices $M_k$, $k\in\mathbb N$. Specifically, we propose two algebraic assumptions on the generalized Jacobian matrix of the proximity operator that can guarantee such equivalence.  

The strong second-order condition given in \cite{mordukhovich2012second,mordukhovich2015full} fully characterizes the strong metric regularity of the subdifferential of the objective function by using the second-order subdifferential based on Mordukhovich's coderivative. However, a drawback of this characterization is the missing link to generalized Jacobians of the stationarity equation (i.e., the underlying KKT system, the natural residual, or the normal map equation) and to the semismooth Newton method. This is of course understandable as the calculation of the generalized Jacobian is nontrivial in general. As a result, connections between the $\SSOSC$ and the generalized Jacobian of the stationarity equation are only known for problems whose stationarity equation has \emph{explicit, closed-form} generalized Jacobians, \cite{sun2006strong,wang2010nonsingularity,guo2015some,liu2019characterizations}.

Since the second subderivative and the derivative of the proximity operator are closely connected, \cite{poliquin1996generalized}, we plan to use the second subderivative rather than the second-order differential to formulate the $\SSOSC$. According to \cite[Theorem 13.24(c)]{rockafellar2009variational}, the standard second-order sufficient condition takes the form $\rd^2\psi(\bar x|0)(h)>0$ for all $h\neq 0$. By the homogeneity and lower semicontinuity of the second subderivative, this condition can be equivalently expressed as  
%
\begin{equation} \label{eq:soc} \rd^2\psi(\bar x|0)(h)\geq \sigma\|h\|^2 \quad \forall~h\in\dom(\rd^2\psi(\bar x|0)) \end{equation}
for some $\sigma>0$. As in the nonlinear programming case, \cite{bonnans2005perturbation,sun2006strong}, we expect the $\SSOSC$ to strenghten \cref{eq:soc} using a proper extension of $\rd^2\psi(\bar x|0)$ on the affine hull of its domain. Such a generalization is not always easy, since, in contrast to the specific applications considered in the literature, the second subderivative of $\vp$ can be too complex to be extended naturally. Hence, we narrow our focus on a special class of functions whose second subderivative provides a natural extension to the affine hull of the critical cone. 

\begin{assumption}
    Let $\bar x \in \crit{(\psi)}$, $\bar v = -\nabla f(\bar x)$, and $\bar z=\bar x+\tau \bar v$ be given. 
\begin{enumerate}[label=\textup{\textrm{(A.\arabic*)}},topsep=0pt,itemsep=0.5ex,partopsep=0ex]
    \item \label{A1} The function $\vp$ is globally prox-regular at $\bar x$ for $\bar v$ with constant $\rho$ and $\tau\rho<1$.  
        \item \label{A2} The function $\vp$ is twice epi-differentiable at  $\bar x$ for $\bar v$ and its second subderivative $\rd^2\vp(\bar x|\bar v)$ has a generalized conic quadratic form, i.e.,
        \[ \rd^2\vp(\bar x|\bar v)(h) = \iprod{h}{Qh} + \iota_{S}(h) \quad \forall~h\in\R^n,\] 
        where $Q \in \R^{n \times n}$ is some symmetric matrix with $\cR(Q)\subseteq \aff(S)$ and $S \subseteq \R^n$ is a closed convex cone. 
\end{enumerate}
\end{assumption}

For simplicity, we will often use the notation $\affS := \aff(S)$. We further note that the condition $\cR (Q)\subseteq \affS$ causes no loss of generality, since we can replace $Q$ by $\Pi_{\affS}Q\Pi_{\affS}$. 

\ref{A1} is a default condition to ensure the single-valuedness and Lipschitz continuity of the proximity operator around $\bar z$ (as mentioned in \Cref{sec:notation}). We assume \ref{A1} throughout this paper. \revise{Here, let us stress again that \ref{A1} holds if $\vp$ is prox-bounded and prox-regular (in the standard sense).} Under assumption \ref{A2} and applying \cite[Exercise 13.18]{rockafellar2009variational}, the second subderivative $\rd^2\psi(\bar x|0)$ satisfies
\be \label{eq:psi-gcf} \rd^2\psi(\bar x|0)(h) = \iprod{h}{[\nabla^2 f(\bar x) + Q]h} + \iota_{S}(h), \ee
and thus, it also has a generalized conic quadratic form. In this case, an extension of $\rd^2\psi(\bar x|0)$ from its domain to the affine hull of its domain can be obtained naturally and directly and we can define the following strong second-order condition.

\begin{defn}[Strong Second-order Sufficient Condition]
    Let $\rd^2\psi(\bar x|0)$ have the form \cref{eq:psi-gcf} for a closed convex cone $S$ and $Q \in \mathbb{S}^n$ with $\cR(Q)\subseteq \aff(S)$. The strong second-order sufficient condition $(\SSOSC)$ holds at $\bar x$ if there is $\sigma>0$ such that:
    \begin{align}
        \label{ssosc}
         \langle h,[\nabla^2f(\bar x)+Q]h\rangle\geq \sigma\|h\|^2 \quad \forall~h\in \aff(S).
    \end{align}
\end{defn}

Motivated by the results in \cite{ouyang2021trust}, we consider a second condition that involves the generalized Jacobian of the proximity operator:
\begin{align}
    \label{cond_gj}
     D\nabla^2f(\bar x)D+{\tau}^{-1} D(I-D)\succeq \sigma D^2 \quad \forall~D\in\partial\proxs(\bar z).
\end{align}
As shown in \cite{ouyang2021trust}, condition \cref{cond_gj} guarantees CD-regularity of the normal map $\oFnor$ at $\bar z$. To illustrate this fact, let $M \in \partial \oFnor(\bar z)$ and $h \in \Rn$ with $Mh = 0$ be arbitrary. Due to $\partial \oFnor(\bar z)h = \mathcal M_{\mathrm{nor}}^\tau(\bar z)h$, there is $D \in \partial \proxs(\bar z)$ such that $Mh = [\nabla^2 f(\bar x)D + {\tau}^{-1}(I-D)]h$.
By \cref{cond_gj}, it then follows $Dh = 0$ and we have $0 = [\nabla^2 f(\bar x)-{\tau}^{-1}I]Dh + {\tau}^{-1}h = {\tau^{-1}}h$. This verifies invertibility of $M$ and the desired CD-regularity. 

We are interested in the relationship between \cref{ssosc} and \cref{cond_gj}. Clearly, a proper characterization of the generalized Jacobians $\partial\proxs$ is integral to establish a connection between \cref{ssosc} and \cref{cond_gj}. The most direct way to verify the equivalence of \cref{ssosc} and \cref{cond_gj} is to calculate $\partial\proxs$, $Q$, and $\aff(S)$ explicitly. Indeed, such a direct computational verification has been the prevalent proof strategy in applications and has been used successfully in, e.g., nonlinear semidefinite programming \cite{sun2006strong,chan2008constraint}, second order cone programs \cite{bonnans2005perturbation,wang2009properties,wang2010nonsingularity}, and Ky Fan $k$-norm cone programming \cite{ding2012introduction}.  
In the following, we propose several abstract conditions on $\partial\proxs$ under which \cref{ssosc} is equivalent to \cref{cond_gj}. Utilizing variational techniques, we then show that these conditions hold for a broad class of strictly decomposable problems.

\subsection{A First Structural Condition for \texorpdfstring{$\partial\proxs$}{partial prox}} \label{sec:first-struc} By \revise{assumptions \ref{A1}--\ref{A2} and} \cite[Proposition 2.3]{ouyang2023partI}, it holds that $0 \preceq D \preceq \frac{1}{1-\tau \rho}I$ for all $D \in \partial\proxs(\bar z)$. In addition, the matrix $I+\tau Q$ is positive definite\footnote{\revise{By \ref{A1} and \cite[Proposition 13.49]{rockafellar2009variational}, the function $h \mapsto T(h):=  \rd^2\vp(\bar x|\bar v)(h)+\rho\|h\|^2$ is convex. Given the representation of $\rd^2\vp(\bar x|\bar v)$ in \ref{A2}, this can be true if and only if $Q+\rho I$ is positive semidefinite. In fact, the convexity of $T$ directly implies $\iprod{y-x}{(Q+\rho I)(y-x)} \geq 0$ for all $x,y \in S$ which yields $\iprod{h}{(Q+\rho I)h} \geq 0$ for all $h \in \aff(S)$. Due to $\cR(Q)\subseteq \aff(S)$, $Q+\rho I$ then has to be positive semidefinite.  
Since $\tau\in (0,1/\rho)$, the positive semidefiniteness of $Q+\rho I$ implies positive definiteness of the matrix $I+\tau Q$.}}. We start with a simple algebraic observation. 

\begin{lem}
    \label{lemma5-19}
    Let $D\in \bS^n$ be given with $0\preceq D \preceq \frac{1}{1-\tau \rho}I$. Then, $D$ can be written as $D=\Pi_{\mathcal R(D)}(I+\tau A)^{-1}\Pi_{\mathcal R(D)}$, where $A\in \bS^n$ satisfies $A + \rho I \succeq 0$ and $\mathcal R(A)\subseteq \mathcal R(D)$.
\end{lem}
\begin{proof}

Let $D = P\Lambda P^\top$ be an eigendecomposition of $D$ with $PP^\top = P^\top P = I$ and eigenvalues $\diag(\lambda) = \Lambda$. Let us assume $\lambda_i > 0$ for $i = 1,\dots, r$ and $\lambda_i = 0$ for all $i = r+1,\dots, n$ and let us further write $P = [U,  V]$ where $U \in \R^{n\times r}$ and $V \in \R^{n \times (n-r)}$. Defining $\Lambda_1 := \mathrm{diag}(\lambda_1,\dots,\lambda_r)$ and 
\[ A := \frac{1}{\tau} P \begin{pmatrix} \Lambda_1^{-1} - I & 0 \\ 0 & 0 \end{pmatrix} P^\top = \frac{1}{\tau} U(\Lambda_1^{-1} - I)U^\top,  \]
it follows $\Pi_{\mathcal R(D)}(I+\tau A)^{-1}\Pi_{\mathcal R(D)} = U\Lambda_1U^\top = D$, $\mathcal R(A)\subseteq \mathcal R(D)$, $\Lambda_1^{-1} \succeq (1-\tau\rho)I$, and $A + \rho I \succeq 0$.
\end{proof}

\cref{lemma5-19} motivates a first structural condition for the generalized Jacobians: 
\begin{equation}
    \label{cond1_gj} \tag{P.1}
    \begin{array}{l} \textit{for every $D\in\partial\proxs(\bar z)$, we have:} \\[1ex] 
    \hspace{1ex} D=\Pi_{\mathcal R(D)}(I+\tau A)^{-1}\Pi_{\mathcal R(D)},\;\; \cR(A)\subseteq \mathcal R(D) \subseteq \affS, \;\; A \succeq \Pi_{\mathcal R(D)}Q\Pi_{\mathcal R(D)}, \end{array}
\end{equation}
%
%
where $\affS = \aff(S)$. The positive definiteness of the matrix $I+\tau Q$ ensures that $(I+\tau A)^{-1}$ \revise{is a bijective mapping} from $\mathcal R(D)$ to $\mathcal R(D)$. We now show that, under condition \cref{cond1_gj}, the strong second-order sufficient condition \cref{ssosc} implies \cref{cond_gj}.

\begin{prop}
    \label{ssoscimcondgj}
    Assume \ref{A1}, \ref{A2}, and \cref{cond1_gj}. Then, \cref{ssosc} implies \cref{cond_gj}.
\end{prop}
\begin{proof}
    By \cref{cond1_gj} and setting $\cS := \cR(D)$, the condition $D\nabla^2f(\bar x)D+\frac{1}{\tau} D(I-D)\succeq \sigma D^2$ can be written as \begin{align*}
         &\Pi_{\cS} (I+\tau A)^{-1}\Pi_{\cS}\nabla^2f(\bar x)\Pi_{\cS}(I+\tau A)^{-1} \Pi_{\cS}
         \\& \hspace{4ex}+\Pi_{\cS}(I+\tau A)^{-1}A(I+\tau A)^{-1}\Pi_{\cS}\succeq \sigma \Pi_{\cS} (I+\tau A)^{-2}\Pi_{\cS},
    \end{align*}
    which is equivalent to
    \[ \langle (I+\tau A)^{-1}h ,(\Pi_{\cS}\nabla^2 f(\bar x)\Pi_{\cS}+A)(I+\tau A)^{-1}h \rangle  \geq \sigma \|(I+\tau A)^{-1}h\|^2 \quad \forall~h\in\cS.  \]
    Notice that since $\cR(A)\subseteq \cS$, $(I+\tau A)^{-1}$ is a bijective mapping from $\cS$ to $\cS$. Hence, replacing $(I+\tau A)^{-1}h$ by $h$, the last condition is further equivalent to:
    \[     \langle h ,(\Pi_{\cS}\nabla^2 f(\bar x)\Pi_{\cS}+A)h \rangle  \geq \sigma \|h\|^2 \quad \forall~h\in\cS.              \]
    Using \cref{cond1_gj} and \cref{ssosc}, it follows $\Pi_{\cS}\nabla^2 f(\bar x)\Pi_{\cS}+A\succeq \Pi_{\cS}\nabla^2 f(\bar x)\Pi_{\cS}+\Pi_{\cS}Q\Pi_{\cS}\succeq \sigma\Pi_{\cS}^2$ which concludes the proof.                          
\end{proof}

Mimicking the proof of \cref{ssoscimcondgj}, we can show that the condition 
\begin{equation}
    \label{cond2_gj} \tag{P.2}
      \Pi_{\affS} (I+\tau Q)^{-1}\Pi_{\affS}\in \partial \proxs(\bar z).
\end{equation}
%
%
ensures the implication ``\cref{cond_gj} $\implies$ \cref{ssosc}''. We will omit an explicit verification here.

\begin{prop}
    \label{condgjimssosc}
    Assume \ref{A1}, \ref{A2}, and \cref{cond2_gj}. Then \cref{cond_gj} implies \cref{ssosc}.
\end{prop}

\subsection{A Weaker Form of \texorpdfstring{\cref{cond1_gj}}{(2.4)}} We now consider a weaker variant of the condition \cref{cond1_gj} that only involves the Bouligand subdifferential $\partial_B\proxs(\bar z)$: 
\begin{equation}
    \label{weakcond1_gj} \tag{P.3}
      \begin{array}{l} \textit{for every $D\in\partial_B\proxs(\bar z)$, we have:} \\[1ex] \hspace{1ex}  D=\Pi_{\mathcal R(D)}(I+\tau A)^{-1}\Pi_{\mathcal R(D)}, \quad \cR(A)\subseteq \mathcal R(D) \subseteq \affS, \quad A \succeq \Pi_{\mathcal R(D)}Q\Pi_{\mathcal R(D)}. \end{array}
\end{equation} 
%
%
Our goal is to prove that \cref{weakcond1_gj} already implies the stronger condition \cref{cond1_gj}. We first present a result on the range of the sum of positive semidefinite matrices.  
\begin{prop}
    Let $\{A_i\}_{i=1,\dots,r}$ be a family of symmetric positive semidefinite matrices and let $a_i>0$, $i = 1,\dots,r$, be given. Then, $\cR(\sum_{i=1}^ra_iA_i)=\sum_{i=1}^r\cR(A_i)$ where the sum of subspaces is defined as $\sum_{i=1}^r S_i :=\{\sum_{i=1}^rx_i:  x_i\in S_i, i = 1,\dots,r\}$.  
\end{prop}
\begin{proof}
    We note that for a symmetric matrix $A$, it follows $\ker(A)=\cR(A)^{\perp}$ and for two positive semidefinite matrices $A$, $B$, we have $\ker(A) \cap \ker(B) = \ker(A+B)$. Therefore by \cite[Propositions 6.26 and 6.34]{bauschke2017convex}, it holds that
    \begin{align*}
        &\cR\left({\sum}_{i=1}^ra_iA_i\right)=\ker\left({\sum}_{i=1}^ra_iA_i\right)^\perp =\left({\bigcap}_{i=1}^r\ker(A_i)\right)^\perp = {\sum}_{i=1}^r\cR(A_i). 
    \end{align*}
\end{proof}

Next, based on the representation $D = \Pi_{\cR(D)}(I+\tau A)^{-1}\Pi_{\cR(D)}$ in \cref{weakcond1_gj}, we discuss the order of positive semidefinite matrices under projections.
\begin{lem}
    \label{lemma2-16}
    Let $L \subseteq \Rn$ be a linear subspace of $\R^n$ and let $B\in \bS^n$ be given with $I+B\succ 0$. Then, it holds that $\Pi_{L} (I+\Pi_{L}B\Pi_L)^{-1}\Pi_{L}\preceq (I+B)^{-1}$.              
\end{lem}
\begin{proof}
    Without loss of generality, let us assume $L=\spa\{e_1,\dots,e_r\}$ and let us write $B$ in block form:
    \[ B=\begin{pmatrix}
       B_1 & B_2 \\
       B_2^\top & B_3 
    \end{pmatrix}, \quad B_1\in\R^{r\times r}, \quad B_2\in\R^{r\times n-r},\quad B_3\in \R^{n-r\times n-r}.         \]
    By the block inverse formula \cite[Equation (0.8.5.6)]{horn2012matrix} and using the Schur complement $C=(I+B_3)-B_2^\top(I+B_1)^{-1}B_2$ for $I+B$, we see that:
 {\small
        \begin{align*}          
            & (I+B)^{-1}=\begin{pmatrix}
                (I+B_1)^{-1}+(I+B_1)^{-1}B_2C^{-1}B_2^\top(I+B_1)^{-1}  &  -(I+B_1)^{-1}B_2C^{-1}  \\
                -C^{-1}B_2^\top (I+B_1)^{-1}  & C^{-1}
             \end{pmatrix}, \\ & \Pi_{L} (I+\Pi_{L}B\Pi_L)^{-1}\Pi_{L}=\begin{pmatrix}
                 (I+B_1)^{-1} & 0 \\
                 0 & 0
             \end{pmatrix}.   
     \end{align*}}
    
    \noindent Notice that $C^{-1}$ is positive definite as principle submatrix of the positive definite matrix $(I+B)^{-1}$. Moreover, we have
    \begin{align*}
        &(I+B)^{-1}-\Pi_{L} (I+\Pi_{L}B\Pi_L)^{-1}\Pi_{L}
        \\ & \hspace{4ex}  =\begin{pmatrix}
            (I+B_1)^{-1}B_2C^{-1}B_2^\top(I+B_1)^{-1} & -(I+B_1)^{-1}B_2C^{-1} 
            \\-C^{-1}B_2^\top (I+B_1)^{-1}  & C^{-1}
        \end{pmatrix}=: E.  
    \end{align*}
    To conclude, it suffices to show that $E$ is positive semidefinite. But taking arbitrary $h=(h_1,h_2)\in\R^n$, it follows $h^\top Eh= \|C^{-\frac12}B_2^\top(I+B_1)^{-1}h_1-C^{-\frac12}h_2\|^2\geq 0$.   
\end{proof}
 
\begin{prop}
    \label{prop3-5}
    Let $B \in \mathbb S^n$ with $I+\tau B \succ 0$ be given and let $\{S_i\}_{i=1,\dots,r}$ be a family of linear subspaces of $\Rn$. For $i=1,\dots,r $ consider $P_i=\Pi_{S_i}(I+\tau A_i)^{-1}\Pi_{S_i}$ with $A_i \in \mathbb S^n$, $\cR(A_i)\subseteq S_i$, and $A_i\succeq \Pi_{S_i} B\Pi_{S_i}$ and set $P=\sum_{i=1}^ra_iP_i$ where $a_i\geq 0$ and $\sum_{i=1}^ra_i=1$. Then, the matrix $P$ can be written as $P=\Pi_{\cR(P)}(I+\tau A)^{-1}\Pi_{\cR(P)}$ with $\cR(A)\subseteq \cR(P)$ and $A\succeq\Pi_{\cR(P)} B\Pi_{\cR(P)}$. In particular, \cref{weakcond1_gj} implies \cref{cond1_gj}.
\end{prop}
\begin{proof}
    We first note that by combining $A_i\succeq \Pi_{S_i} B\Pi_{S_i} $ and $I+\tau B \succ 0$, we can infer that each $I+\tau A_i$ is positive definite. We prove \cref{prop3-5} by induction. The base case $r=1$ is trivial. Let us now consider $r=2$. In the case $a_1=0$ or $a_2=0$ nothing needs to be shown, so we can assume $a_1,a_2>0$. We then have $\cR(P)=S_1+S_2$ and setting $\tilde B :=\Pi_{S_1+S_2}B\Pi_{S_1+S_2}$, it follows:
    \[  A_1\succeq \Pi_{S_1}\tilde B\Pi_{S_1} \quad \text{and} \quad A_2\succeq \Pi_{S_2}\tilde B\Pi_{S_2}.    \] 
    Utilizing \cref{lemma2-16}, we obtain 
    \[  \Pi_{S_1}(I+\tau A_1)^{-1}\Pi_{S_1}\preceq \Pi_{S_1}(I+\tau \Pi_{S_1}\tilde B\Pi_{S_1})^{-1}\Pi_{S_1} \preceq (I+\tau \tilde B)^{-1}                 \]
    and $ \Pi_{S_2}(I+\tau A_2)^{-1}\Pi_{S_2}\preceq (I+\tau \tilde B)^{-1}$ which implies $P\preceq (I+\tau \tilde B)^{-1}$. By \cref{lemma5-19}, we can write $P$ as $P=\Pi_{S_1+S_2}(I+\tau A)^{-1}\Pi_{S_1+S_2}$ for some $A$ with $\cR(A) \subseteq \cR(P) = S_1+S_2$. This yields $\Pi_{S_1+S_2}(I+\tau A)^{-1}\Pi_{S_1+S_2}\preceq (I+\tau \tilde B)^{-1}$. Since both $\cR(A)$ and $\cR(\tilde B)$ are subsets of $S_1+S_2$, this condition is true if and only if\footnote{\revise{Let us set $L=S_1+S_2$. If $\cR(A),\cR(\tilde B)\subseteq L$, then we have $(I+\tau A)^{-1}=\Pi_L(I+\tau A)^{-1}\Pi_{L}+\Pi_{L^\perp}(I+\tau A)^{-1}\Pi_{L^\perp}=\Pi_L(I+\tau A)^{-1}\Pi_{L}+\Pi_{L^\perp}$ and $(I+\tau \tilde B)^{-1}$ can be expressed in a similar way. (This can be formally shown, e.g., by following the block decomposition used in the proof of \cref{lemma2-16}). Therefore, $ (I+\tau A)^{-1}\preceq (I+\tau \tilde B)^{-1}$ if and only if $ \Pi_L(I+\tau A)^{-1}\Pi_L\preceq \Pi_L(I+\tau \tilde B)^{-1}\Pi_L$, which is further equivalent to  $\Pi_L(I+\tau A)^{-1}\Pi_L\preceq \Pi_L(I+\tau \tilde B)^{-1}\Pi_L+\Pi_{L^\perp}=(I+\tau\tilde B)^{-1}$.}}:
    \[        (I+\tau A)^{-1}\preceq (I+\tau \tilde B)^{-1} \quad \iff \quad A\succeq \tilde B,    \]
    which proves the case $r=2$. Now suppose the conclusion is true for $1,\dots,r-1$. There is no loss of generality, if we assume $a_i>0$ for all $i$. Let us set $\tilde P=\frac{1}{1-a_r}\sum_{i=1}^{r-1}a_iP_i$. Using induction for $\tilde P$ and applying the last steps for $P=(1-a_r)\tilde P+a_rP_r$, we can conclude the first part of the proof. The remaining assertion in \cref{prop3-5} follows from $\partial \proxs(\bar z)=\conv(\partial_B\proxs(\bar z))$ and the Carath{\'e}odory theorem, cf. \cite{rockafellar2009variational}. 
\end{proof}

\subsection{Second-order Variational Formulations for \texorpdfstring{\cref{cond2_gj}}{(2.5)} and \texorpdfstring{\cref{weakcond1_gj}}{(2.6)}} 

Due to the close relationship between the generalized Jacobian of the proximity operator $\proxs$ and the second subderivative of $\vp$,
we now plan to convert \cref{cond2_gj} and \cref{weakcond1_gj} into corresponding second-order variational conditions. The main tool is \cite[Lemma 3.13]{ouyang2023partI}, which allows connecting the limit of \revise{$\D\proxs$} and the epi-limit of the second subderivative of $\vp$ with each other.
\begin{prop} \label{proposition:v1-p3}
    Let \ref{A1} and \ref{A2} be satisfied. Assume further that for all $\gph(\partial\vp) \ni (x^k,v^k)\to(\bar x,\bar v)$ such that $\vp$ is twice epi-differentiable at $x^k$ for $v^k$ and $\rd^2\vp(x^k|v^k)$ is generalized quadratic, it holds that:
    \begin{align}
        \label{cond_sosub1} \tag{V.1}
        ({\eliminf}_{k\to\infty}\,\rd^2\vp(x^k|v^k))(w) \geq \langle w, Qw\rangle+\iota_{\aff(S)}(w) \quad \forall~w \in \Rn.
    \end{align}
    Then, condition \cref{weakcond1_gj} is satisfied.
\end{prop}
\begin{proof}
    Let $D\in\partial_B\proxs(\bar z)$ be arbitrary. Due to \cite[Proposition 2.3]{ouyang2023partI}, we have $0 \preceq D \preceq \frac{1}{1-\tau\rho}I$ and applying \cref{lemma5-19}, it holds that $D=\Pi_{\cR(D)}(I+\tau A)^{-1}\Pi_{\cR(D)}$ with $A \in \mathbb S^n$, $A+\rho I\succeq 0$, and $\cR(A)\subseteq \cR(D)$. By definition, there is $z^k\to \bar z$ such that $\proxs$ is differentiable at $z^k$ and $\revise{\D\proxs(z^k)}\to D$. Let us define $x^k := \proxs(z^k)$ and $v^k := \frac{1}{\tau}(z^k-x^k)$. Then, it follows $z^k=x^k+\tau v^k$ and $v^k\in\partial \vp(x^k)$. Applying \cite[Lemma 3.13]{ouyang2023partI}, the condition $\revise{\D\proxs(z^k)}\to D$ implies that $\vp$ is twice epi-differentiable at $x^k$ for $v^k$, $\rd^2\vp(x^k|v^k)$ is generalized quadratic (for all $k$), and we have 
    \[  ({\elim}_{k\to\infty}\,\rd^2\vp(x^k|v^k))(w)= \langle w,Aw\rangle+\iota_{\cR(D)}(w) \quad  \forall~ w\in \R^n.                          \]
    By assumption, we know that:
    \[  \langle w,Aw\rangle+\iota_{\cR(D)}(w)\geq \langle w,Qw \rangle+\iota_{\aff(S)}(w) \quad \forall~w\in\R^n.             \]
    This yields $ \Pi_{\cR(D)}Q\Pi_{\cR(D)} \preceq A$ and $\cR(A)\subseteq \cR(D)\subseteq \aff(S)$ which proves \cref{weakcond1_gj}.
\end{proof}

\begin{prop}
    \label{prop2-10}
    Let \ref{A1} and \ref{A2} hold. Assume further that there exists a sequence $\gph(\partial\vp)\ni(x^k,v^k)\to(\bar x,\bar v)$ such that $\vp$ is twice epi-differentiable at $x^k$ for $v^k$, $\rd^2\vp(x^k|v^k)$ is generalized quadratic, and we have
    \begin{align}
        \label{cond_sosub2} \tag{V.2}
        ({\elim}_{k\to\infty}\,\rd^2\vp(x^k|v^k))(w)=\langle w, Qw\rangle+\iota_{\aff(S)}(w) \quad \forall~w \in \Rn.
    \end{align}
    Then, condition \cref{cond2_gj} is satisfied.
\end{prop}
\begin{proof}
     This is a direct consequence of \cite[Lemma 3.13]{ouyang2023partI}.
\end{proof}


\section{Second-order Variational Properties of \texorpdfstring{$C^2$}{C2}-Strictly Decomposable Functions} 
\label{sec:eq_ssonc}
In this section, we introduce the concept of $C^2$-strict decomposability. Strictly decomposable functions satisfy condition \ref{A2} and we show that  \cref{cond_sosub1} and \cref{cond_sosub2} will hold under additional mild geometric assumptions.

\begin{defn}[Decomposable functions] A function $ \varphi: \mathbb{R}^{n} \rightarrow \Rex $ is called $ C^{\ell} $-decomposable, $ \ell \in \mathbb{N} $, at a point $ \bar{x} \in \dom(\varphi)$, if there is an open neighborhood $ U $ of $ \bar{x} $ such that
    $$
    \varphi(x)=\varphi(\bar{x})+\varphi_{d}(F(x)) \quad \forall~x \in U,
    $$
    and the functions $ \varphi_{d} $ and $ F $ satisfy:
  \begin{enumerate}[label=\textup{\textrm{(\roman*)}},topsep=0pt,itemsep=0ex,partopsep=0ex]
      \item $ F: \mathbb{R}^{n} \rightarrow \mathbb{R}^{m} $ is $\ell$-times continuously differentiable on $ U $ and $ F(\bar x)=0 $.
      \item The mapping $ \varphi_{d}: \mathbb{R}^{m} \rightarrow \Rex $ is convex, lsc, positively homogeneous, and proper, or equivalently, $\vp_d=\sigma_{\pvpd}$ for some closed convex set $\pvpd$.  
      \item Robinson's constraint qualification holds at $ \bar{x} $ :
    $$
    0 \in \operatorname{int}\left\{F(\bar{x})+\revise{\D F(\bar{x})} \mathbb{R}^{n}-\dom(\varphi_{d})\right\}=\operatorname{int}\left\{\revise{\D F(\bar{x})} \mathbb{R}^{n}-\dom(\varphi_{d})\right\} .
    $$
  \end{enumerate}
  We say that $ \varphi $ is $ C^{\ell} $-strictly decomposable at $ \bar{x} $ for $\bar\lambda$ if $ \varphi $ is $ C^{\ell} $-decomposable at $ \bar{x} $ and if, in addition, the strict condition
  \be \label{eq:cq-strict} \revise{\D F(\bar{x})} \mathbb{R}^{n}-N_{\pvpd}(\bar\lambda)=\mathbb{R}^{m} \ee
  is satisfied at $ \bar{x} $ for some $ \bar\lambda \in \pvpd $. We say that $ \varphi $ is $ C^{\ell} $-fully decomposable at $ \bar{x} $ if the strict condition is replaced by the nondegeneracy condition $\revise{\D F(\bar{x})} \mathbb{R}^{n}+\operatorname{lin}(N_{\pvpd}(\bar\lambda))=\mathbb{R}^{m}$.
\end{defn}

The concept of decomposability was first proposed by Shapiro \cite{shapiro2003class} and serves as a functional analogue to cone reducible sets. Examples of $C^2$-strictly decomposable functions include, e.g., polyhedral and piecewise-linear functions, the total variation and the group-lasso norm, and the Ky-Fan $k$-norm. For more examples and background, we refer to \cite{shapiro2003class,milzarek2016numerical,ouyang2023partI}.

By \cite[Proposition 3.7]{ouyang2023partI}, if $\vp$ is $C^2$-strictly decomposable at $\bar x$ for $\bar \lambda \in \pvpd$, then $\vp$ is properly twice epi-differentiable at $\bar x$ for $\bar v=\revise{\D F(\bar x)^\top}\bar \lambda \in \partial\vp(\bar x)$ with 
\be \label{eq:sub-strict}    \rd^2\vp(\bar x|\bar v)(w)=\langle \bar\lambda, \revise{\revise{\D^2F(\bar x)}}[w,w] \rangle+\iota_{\revise{\D F(\bar x)^{-1}}N_{\pvpd}(\bar\lambda)}(w) \quad \forall~w \in \Rn.           \ee
Here, $\bar\lambda\in\pvpd$ is uniquely determined by $\bar v \in \partial\vp(\bar x)$ thanks to \cite[Proposition 3.6]{ouyang2023partI}. Defining $H(x,\lambda) := \sum_{i=1}^m \lambda_i \nabla^2 F_i(x)$, assumption \ref{A2} is then clearly satisfied with $S=\revise{\D F(\bar x)^{-1}}N_{\pvpd}(\bar\lambda)$ and $Q=\Pi_{\affS} H(\bar x,\bar\lambda)\Pi_{\affS}$.  
In the following, we verify that the conditions \cref{cond_sosub1} and \cref{cond_sosub2} hold for $C^2$-strictly decomposable functions. As in \cite{ouyang2023partI} and as mentioned, this requires an additional geometric assumption on the support set $\pvpd$. 

\revise{Before doing so, we provide an example for a function that is $C^2$-strictly but not $C^2$-fully decomposable---underlining the generality and relevance of strict decomposability. 
\begin{example} \label{example:strict-decomp-not-full}
We consider the univariate function $\vp:\R\to\R$ defined as $\vp(x):=\max\{0,x,x+x^5\sin(1/x)\}$, where the term $x^5\sin(1/x)$ is treated as $0$ when $x=0$. We set $C:=\conv\{(0,0)^\top, (0,1)^\top, (1,0)^\top\} \subseteq \R^2$ and $F:\R\to \R^2$, $F(x) := (x,x+x^5\sin(1/x))^\top$.
Then, using \cite[Theorem 3.3.2]{HirLem01}, it follows $\sigma_C(x,y) = \max\{0,x,y\}$ and we can infer $\vp=\sigma_C\circ F$.
%
%
Moreover, due to $\revise{\D F(0)}=(1,1)^\top$, we have
%
\[    \D F(0)\R-N_C(0)=\begin{pmatrix}
    1 \\
    1
\end{pmatrix}\R-\R^2_-=\R^2.       \] 
This shows that $\vp$ is $C^2$-strictly decomposable at $0$ for $0\in C$.
Next, we verify that $\vp$ is not $C^2$-fully decomposable at $0$. Our proof builds on the following observation: 

\begin{enumerate}[label=\textup{$\bullet$},topsep=4pt,itemsep=0ex,partopsep=0ex]
\item Let $\vp:\R\to\R$ be $C^2$-fully decomposable at $0$. Then, $\vp$ can be represented as $\max\{\vp_1,\vp_2\}$ locally around $0$ with $\vp_1,\vp_2$ being $C^2(\R)$.
\end{enumerate}
This result is of somewhat independent interest and its derivation is deferred to \cref{app:example}. We now verify that $\vp$ cannot be represented in the form $\max\{\vp_1,\vp_2\}$ locally around $0$ with $\vp_1,\vp_2\in C^2(\R)$. Suppose---in order to reach a contradiction---that such a representation exists. Noticing $\vp^\prime_+(0) = \lim_{t \downarrow 0} \vp(t)/t = 1$ and $\vp^\prime_-(0)= \lim_{t \uparrow 0} \vp(t)/t = 0$, we may assume without loss of generality that $\vp_1'(0)=1$ and $\vp_2'(0)=0$. Therefore, locally around $0$, it follows $\vp = \vp_1$ when $x>0$ and $\vp = \vp_2$ when $x<0$ (by considering the Taylor expansions of $\vp_1$ and $\vp_2$ around $0$). In particular, $\vp$ has to be $C^2$ on $(0,\epsilon)$ for some $\epsilon>0$. However, using a direct calculation, we have $\vp(x)=x+\max\{0,x^5\sin(1/x)\}$ on $(0,1)$, which is not $C^2$ on any $(0,\epsilon)$, $\epsilon>0$. This yields a contradiction and proves that $\vp$ is not $C^2$-fully decomposable at $0$.
\end{example}
}

\subsection{Verifying \texorpdfstring{\cref{cond_sosub1}}{(2.7)} and the \texorpdfstring{$\sotp$}{OTP}-Property}
\begin{defn}
    \label{def-sotp}
    We say that a closed convex set $\pvpd$ has the outer second-order tangent path ($\sotp$-)property at $\bar\lambda \in \pvpd$, if there is a constant $M>0$ and a neighborhood $V$ of $\bar\lambda$ such that for all $\lambda\in \pvpd\cap V$, it holds that:
 \begin{enumerate}[label=\textup{\textrm{(\roman*)}},topsep=0pt,itemsep=0ex,partopsep=0ex]
      \item We can select a linear subspace $L(\lambda)\subseteq \lin(T_{\pvpd}(\lambda))$ such that $L(\lambda)\to \lin(T_{\pvpd}(\bar\lambda))$ as $\lambda\to\bar\lambda$;
      \item For all $p\in L(\lambda)$ with $\|p\|=1$, we have $\dist(0,T^{2}_\pvpd(\lambda,p))<M$. 
    \end{enumerate} 
    If this holds for all $\bar\lambda\in \pvpd$, then we say that $\pvpd$ has the $\sotp$-property.
\end{defn}

The $\sotp$-property is a weaker variant of the uniform second-order tangent path ($\usotp$-)property recently introduced in \cite[Definition 3.16]{ouyang2023partI}. Compared to the $\usotp$-property, uniform length of the second-order path along the direction $p \in L(\lambda)$ is not required and we only need the curvature of the second-order outer tangent set to be uniformly bounded rather than the inner one. Since we only need to control the second subderivative $\rd^2\vp$, which is defined as limit of second-order difference quotients, uniformness of the tangent paths is not necessary. In addition, since $\vp$ is assumed to be twice epi-differentiable at $\bar x$ for $\bar v$ in \cref{cond_sosub1} and \cref{cond_sosub2}, the inner second-order tangent set can be replaced by the outer one (as illustrated in the proof of \cref{lemma5-10}). 

Since the $\sotp$-property is implied by the $\usotp$-property, all sets having the $\usotp$-property also have $\sotp$-property. In particular, this includes  $C^2$-pointed and $C^2$-cone reducible sets, cf. \cite{ouyang2023partI}.


Next, we construct two counterexamples to illustrate that both conditions in \cref{def-sotp} are necessary for the implication ``\cref{ssosc} $\implies$ \cref{cond_gj}''.

\begin{figure}[t]
	\centering
	\subfigure[Visualization for \cref{exam2-11}]
	{\includegraphics[width=6.8cm]{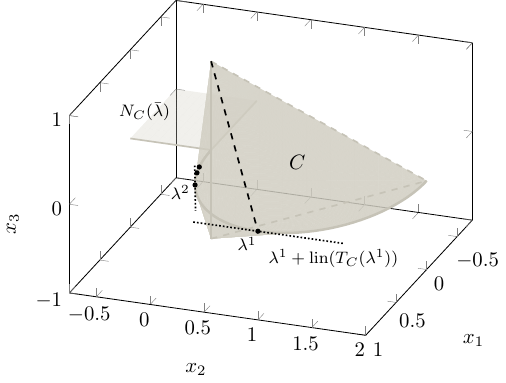}} \hspace{1.2ex}
	\subfigure[Visualization for \cref{exam2-12}]
	{\includegraphics[width=5.8cm]{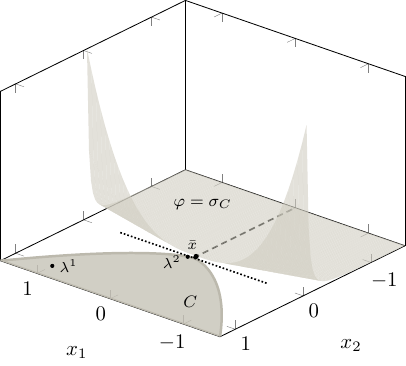}}
	\caption{Illustration of the sets and constructions in \cref{exam2-11} and \cref{exam2-12}. \label{fig:examples}}
\end{figure}

\begin{example}[Necessity of condition $\mathrm{(i)}$]
    \label{exam2-11}
We consider:
    \begingroup
    \allowdisplaybreaks
    \begin{align*}
        &C_1= \{(0,0,x_3): x_3 \in[-1,1]\}, \quad C_2=   \{(x_1,x_2,0):x_1^2+(x_2-1)^2=1, \, x_1 \geq 0\}, 
        \\& C=\overline{\mathrm{conv}}(C_1\cup C_2),\quad                
          \vp_d(x)=\sigma_C(x), \quad F=I,\quad \vp=\vp_d, \quad \text{$f$ free},
    \end{align*}
    \endgroup
    and $\bar x = \bar \lambda = (0,0,0)^\top$. Then, it holds that $\bar\lambda\in\partial \vp(0) = C$ and it follows $N_C(\bar\lambda)=\{(x_1,x_2,0)^\top: x_1,x_2\leq 0\}$. By construction, $\vp$ is $C^2$-fully decomposable at $\bar x$ with decomposition pair $(\vp_d, F)$. Hence, using \cref{eq:sub-strict} with $\bar v = \revise{\D F(\bar x)^\top} \bar\lambda = 0$, we have 
    \[     \rd^2\vp(\bar x|0)(h)=\langle h, H(\bar x,\bar \lambda) h \rangle +\iota_{N_C(\bar\lambda)}(h)=\iota_{N_C(\bar\lambda)}(h) \quad \forall~h \in \Rn           \]
    and the strong second-order sufficient condition \cref{ssosc} is equivalent to:
    \begin{align}
        \label{exam2-11eq1}
        \iprod{h}{\nabla^2f(\bar x)h} \geq \sigma\|h\|^2 \quad \forall~h\in \aff(N_C(\bar\lambda)) = \spa\{(1,0,0)^\top,(0,1,0)^\top\}.   
    \end{align}
    Next, let us define $x^k :=\frac{1}{k}(\frac{1}{k},-\sqrt{1-1/k^{2}},0)^\top$ and $\lambda^k:=(\frac{1}{k},1-\sqrt{1-{1}/{k^2}},0)^\top$. We first claim that $\lambda^k\in\partial\vp(x^k)=\partial\vp_d(x^k) = \{\lambda \in C: \iprod{\lambda}{x^k} = \vp_d(x^k)\} $, which requires to compute $\vp_d$. Let $x \in K := \{x \in \R^3: x_1 > 0, \, x_2<0, \, \text{and}\, |x_3|<\sqrt{x_1^2+x_2^2}+x_2\}$ be given. Applying \cite[Theorem 3.3.2]{HirLem01}, we can calculate:
    \begin{align*}
        \sigma_C(x) =\max\{\sigma_{C_1}(x),\sigma_{C_2}(x)\}
        =\max\{|x_3|,\sqrt{x_1^2+x_2^2}+x_2\}=\sqrt{x_1^2+x_2^2}+x_2.
    \end{align*}
   This expression is valid for all $x \in K$ and hence, $\vp$ is twice continuously differentiable at such $x$. The gradient and Hessian of $\vp$ at $x$ are given by: 
   %
   \[ \nabla \vp(x) = \frac{1}{n(x)} \begin{pmatrix} {x_1} \\ n(x) + {x_2} \\ 0 \end{pmatrix}, \quad
        \nabla^2\vp(x)=\frac{1}{n(x)^3}\begin{pmatrix}
             {n(x)^2}-{x_1^2} & -{x_1x_2} & 0 \\
             -{x_1x_2} & {n(x)^2}-{x_2^2} & 0  \\
             0 & 0 & 0
            \end{pmatrix}, \]
            %
            %
            where $n(x) := \sqrt{x_1^2+x_2^2}$. We then obtain $x^k \in K$, $\lambda^k \in C_2$, and $\vp_d(x^k)=\langle \lambda^k,x^k \rangle$ for all $k$, which verifies $\lambda^k\in\partial\vp(x^k)$. In addition, using the implicit function theorem, $\proxs$ is differentiable at $z^k=x^k+\tau\lambda^k$ with $\revise{\D\proxs(z^k)}=(I+\tau\nabla^2\vp(x^k))^{-1}$. (Notice that $\lambda^k\in\partial\vp(x^k)$ implies $x^k = \proxs(z^k)$). It holds that $z^k\to (0,0,0)^\top$ and 
            %
   %
   \[ \revise{\D\proxs(z^k)}=(I+\tau\nabla^2\vp(x^k))^{-1} \to \mathrm{diag}(0,1,1)   
   =: \bar D \quad \text{as} \; k \to \infty.\]
   Thus, the second-order condition \cref{cond_gj} for $\bar D \in \partial\proxs(0)$ requires:
    \[    \iprod{h}{\nabla^2f(\bar x)h} \geq \sigma\|h\|^2 \quad \forall~h\in \spa\{(0,1,0)^\top,(0,0,1)^\top\}.         \]
    Of course, this is not implied by condition~\cref{exam2-11eq1}. Furthermore, \cref{cond1_gj} and \cref{cond_sosub1} fail to hold. In this example, condition $\mathrm{(i)}$ in the definition of the $\sotp$-property is not satisfied. We actually have $\lin(T_C(\bar\lambda))=\{(0,0,t):t\in \R\}$ but $\lin(T_C(\lambda^k))\to \{(t,0,0):t\in \R\}$.
\end{example}

\begin{example}[Necessity of condition $\mathrm{(ii)}$]
    \label{exam2-12} We consider \cite[Example 3.21]{ouyang2023partI}:
    \[ C =\{ x \in \R^2: x_2\geq \tfrac{2}{3}|x_1|^\frac32\},\quad F=I,\quad \vp(x)=\vp_d(x)=\sigma_C(x), \quad \text{$f$ free} \]
    with $\bar x = \bar\lambda =(0,0)^\top$. Then, it holds that $\bar\lambda\in\partial \vp(0)$,
    \[   N_C(\bar\lambda)=\{(0,t):t\leq 0\},\quad \text{and} \quad \lin(T_C(\bar\lambda))=\{(t,0):t\in\R\}.        \]
    In addition, we have $\lin(T_C(\lambda))\to\lin(T_C(\bar\lambda))$ for $C \backslash \ri(C)\ni\lambda\to\bar\lambda$. Therefore, the first condition in the definition of the $\sotp$-property is satisfied. However, the second condition does not hold since it can be shown that the outer second-order tangent sets at $\bar\lambda$ for the directions $(1,0)$ and $(-1,0)$ are both empty. Similar to \cref{exam2-11} and by construction of $\vp$, it follows $\rd^2\vp(\bar x|\bar\lambda)(h)=\iota_{N_C(\bar\lambda)}(h)$ for all $h$. Hence, condition \cref{ssosc} is again equivalent to:
    \[    \iprod{h}{\nabla^2f(\bar x)h} \geq \sigma \|h\|^2 \quad \forall~h\in \aff(N_C(\bar\lambda)) = \spa\{(0,1)^\top\}.                   \]
    Let us set $x^k=(\frac{1}{k^3},-\frac{1}{k})^\top$ and $\lambda^k=(\frac{1}{k^{4}},\frac{1}{k^6})^\top$. As calculated in \cite[Example 3.21]{ouyang2023partI}, it holds that $\lambda^k\in\partial\vp(x^k)$ and 
    \[   \nabla^2\vp(x^k)=  2\begin{pmatrix}
        \frac{1}{k}  & -\frac{1}{k^3}  \\
        -\frac{1}{k^3}  & \frac{1}{k^5}
     \end{pmatrix}\to 0.    \]
    Setting $z^k=x^k+\tau \lambda^k$, we have $\proxs(z^k)=x^k$, $z^k\to\bar z=\bar x$ and we further obtain:
    \[      \revise{\D\proxs(z^k)}=(I+\tau\nabla^2\vp(x^k))^{-1}\to I\in\partial\proxs(\bar z).                       \]
    Thus, \cref{cond_gj} requires $\nabla^2 f(\bar x)\succeq \sigma I$. Consequently, condition \cref{cond_gj} is stronger than \cref{ssosc}. Moreover, condition \cref{cond1_gj} fails, which implies that \cref{cond_sosub1} also fails. 
\end{example}

Recall that condition \ref{A1} is assumed throughout this paper. We now list \revise{our} main assumptions to ensure \cref{cond_sosub1} and ``\cref{ssosc} $\implies$ \cref{cond_gj}'':
\begin{assum}
    We consider the conditions: 
\begin{enumerate}[label=\textup{\textrm{(B.\arabic*)}},topsep=0pt,itemsep=0.5ex,partopsep=0ex]
        \item \label{B1}  $\vp$ is $C^2$-strictly decomposable at $\bar x$ for $\bar\lambda$ with decomposition pair $(\vp_d\equiv\sigma_\pvpd,F)$.
        \item \label{B2} In addition to \ref{B1}, suppose that $\pvpd$ has the $\sotp$-property at $\bar\lambda$.
    \end{enumerate}
\end{assum}

As in \cite{ouyang2023partI}, our first task is to exploit the postulated geometric properties of the support set $\pvpd$ in the estimation of the second subderivative of $\vp$. 

The following technical lemma is our main tool. 
\begin{lem}
    \label{lemma5-10}
    Suppose that $\vp$ has the decomposition $\vp = \sigma_{\pvpd}\circ F$ around $x \in \dom(\vp)$ and let $\lambda \in \partial\sigma_{\pvpd}(F(x))$, $ p\in T_\pvpd(\lambda)$, $w\in \R^n$ be given with $\langle p,\revise{\D F(x)}w\rangle>0$ and $\langle p,F(x)\rangle=0$. Furthermore, assume that $\dist(0,T^{2}_\pvpd(\lambda,p))<M$ and $\vp$ is twice epi-differentiable at $x$ for $\revise{\D F(x)^\top} \lambda$. Then, it holds that
    \[  \rd^2\vp(x|\revise{\D F(x)^\top}\lambda)(w)\geq   \frac{|\langle p,\revise{\D F(x)}w \rangle|^2}{M\|F(x)\|}+\iprod{w}{H(x,\lambda)w     },       \]
    where the right-hand side is understood as $\infty$ if $F(x)=0$ \revise{and we recall $H(x,\lambda) = \sum_{i=1}^m \lambda_i \nabla^2 F_i(x)$}.
\end{lem}
\begin{proof}
The proof of \cref{lemma5-10} is inspired by the proof of \cite[Lemma 3.23]{ouyang2023partI}. The condition $\dist(0,T^{2}_\pvpd(\lambda,p))< M$ implies that there is a sequence $\{\xi^k\}_k$ such that
    \[             \xi^k=\lambda+s_kp+\frac{1}{2}s_k^2r^k, \quad \xi^k \in \pvpd, \quad s_k\downarrow 0, \quad \|r^k\|\leq M.               \]
    The twice epi-differentiability of $\vp$ at $x$ for $\revise{\D F(x)^\top}\lambda$ says that for any choice of $t_k\downarrow0$, we can find $w^k\to w$ such that
    \[    \rd^2\vp(x|\revise{\D F(x)^\top}\lambda)(w)= \lim_{k\to\infty} \Delta_{t_k}^2  \;\! \vp(x|\revise{\D F(x)^\top}\lambda)(w^k). \] 
    %
    %
    We now select $t_k$ as follows:
   \[     t_k= \begin{cases} s_k^2 & \text{if } F(x)=0, \\ \frac{Ms_k\|F(x)\|}{\langle p,\revise{\D F(x)}w\rangle}  & \text{if } F(x)\neq 0. \end{cases} \] 
   In particular, in the case $F(x)\neq 0$, this choice implies 
   \begin{align}
    \label{formulask}
    s_k=\frac{t_k\langle p,\revise{\D F(x)}w\rangle}{M\|F(x)\|}.       
   \end{align}
   Noticing $\partial\sigma_{\pvpd}(F(x)) = \{\mu \in \pvpd: \iprod{\mu}{F(x)} = \sigma_{\pvpd}(F(x))\}$, we now have:
   \begin{align*}
    & \Delta_{t_k}^2  \;\! \vp(x|\revise{\D F(x)^\top}\lambda)(w^k) = \frac{\sigma_{\pvpd}(F(x+t_kw^k))-\sigma_{\pvpd}(F(x))-t_k\langle \revise{\D F(x)^\top}\lambda,w^k\rangle}{\frac12t_k^2}\\
    & \hspace{6ex} \geq \frac{\langle \xi^k,F(x+t_kw^k)\rangle-\langle \lambda,F(x)\rangle-t_k\langle \revise{\D F(x)^\top}\lambda,w^k\rangle}{\frac12t_k^2} \\
    & \hspace{6ex} = \frac{\langle s_kp+\frac{1}{2}s_k^2r^k,F(x+t_kw^k)\rangle }{\frac12t_k^2}+\frac{\langle \lambda,F(x+t_kw^k)-F(x)-t_k \revise{\D F(x)}w^k\rangle}{\frac12t_k^2} \\
    & \hspace{6ex} =\frac{\langle s_kp+\frac{1}{2}s_k^2r^k,F(x)+t_k\revise{\D F(x)}w^k+\frac{1}{2}t_k^2 E(t_k,w^k)\rangle}{\frac12t_k^2}+\iprod{\lambda}{E(t_k,w^k)},  
   \end{align*}
   where $E(t_k,w^k) := \int_0^1\revise{\D^2F(x+s t_kw^k)}[w^k,w^k] \,\mathrm{d}s \in \Rm$. 
   Utilizing the assumption $\iprod{p}{F(x)} = 0$, we further obtain 
   \begin{align*}
    &\hspace{-4ex}\langle s_kp+\tfrac{1}{2}s_k^2r^k,F(x)+t_k\revise{\D F(x)}w^k+\tfrac{1}{2}t_k^2 E(t_k,w^k)\rangle\\
    &=s_kt_k\langle p,\revise{\D F(x)}w^k\rangle+\tfrac{1}{2}s_k^2\iprod{r^k}{F(x)}+\tfrac12s_k^2 t_k \iprod{r^k}{\revise{\D F(x)}w^k}+o(t_k^2)\revise{.}  
   \end{align*}
   Hence, if $F(x)=0$, it follows:
   \[ \langle s_kp+\tfrac{1}{2}s_k^2r^k,F(x)+t_k\revise{\D F(x)}w^k+\tfrac{1}{2}t_k^2 E(t_k,w^k)\rangle = t_k^{\frac32}\langle p,\revise{\D F(x)}w^k\rangle+\mathcal O(t_k^2)                                  \]
   and due to $\langle p,\revise{\D F(x)}w^k\rangle\to \langle p,\revise{\D F(x)}w\rangle>0$, we can infer
   \begin{align*}
    \Delta_{t_k}^2  \;\! \vp(x|\revise{\D F(x)^\top}\lambda)(w^k) &\geq \frac{t_k^{\frac32}\langle p,\revise{\D F(x)}w^k\rangle+\mathcal O(t_k^2)}{\frac12 t_k^2}+\iprod{\lambda}{E(t_k,w^k)}  \to\infty.    
   \end{align*}
   This proves the conclusion when $F(x)=0$. Next, let us consider the case $F(x)\neq 0$.
   Then, by \cref{formulask}, we have:
   \begin{align*}
    &\hspace{-4ex}\langle s_kp+\tfrac{1}{2}s_k^2r^k,F(x)+t_k\revise{\D F(x)}w^k+\tfrac{1}{2}t_k^2 E(t_k,w^k)\rangle\\
    &=s_kt_k\langle p,\revise{\D F(x)}w^k\rangle+\frac{1}{2}s_k^2\iprod{r^k}{F(x)}+o(t_k^2) \\
    &\geq s_k \left[t_k\langle p,\revise{\D F(x)}w^k \rangle-\frac{1}{2}s_k\|r^k\|\|F(x)\|\right]+o(t_k^2)\\
    &=\frac{t_k^2}{M\|F(x)\|}\left[ \langle p,\revise{\D F(x)}w^k \rangle \langle p,\revise{\D F(x)}w\rangle-\frac{\|r^k\|}{2M}\langle p,\revise{\D F(x)}w \rangle^2 \right] +o(t_k^2).
   \end{align*}
   Therefore, taking the limit $k \to \infty$, it holds that
   \begin{align*}
    &\lim_{k\to\infty}\frac{\langle s_kp+\frac{1}{2}s_k^2r^k,F(x)+t_k\revise{\D F(x)}w^k+\frac{1}{2}t_k^2 E(t_k,w^k)\rangle}{\frac{1}{2}t_k^2} \geq \frac{|\langle p,\revise{\D F(x)}w\rangle|^2}{M\|F(x)\|}.
   \end{align*}
   Combining the previous steps and calculations, this yields
   \begin{align*}
    \rd^2\vp(x|\revise{\D F(x)^\top}\lambda)(w) 
    &\geq \frac{|\langle p,\revise{\D F(x)}w\rangle|^2}{M\|F(x)\|}+\iprod{w}{H(x,\lambda)w}, 
   \end{align*}
   which finishes the proof.
\end{proof}

Thanks to \cref{lemma5-10}, we can now establish the key result of this subsection.

\begin{thm}
    \label{thm3-3}
    Let \ref{B1} be satisfied and suppose there is a sequence $\gph(\partial\vp)\ni(x^k,v^k)\to(\bar x,\bar v)$ such that $\vp$ is twice epi-differentiable at $x^k$ for $v^k$ for all $k \in \mathbb N$. Assume further that there exist a constant $M > 0$ and $\{\lambda^k\}_k$ with $v^k=\revise{\D F(x^k)^\top}\lambda^k$ and $\lambda^k\in\partial\sigma_{\pvpd}(F(x^k))$ such that the following conditions hold:
\begin{enumerate}[label=\textup{\textrm{(\roman*)}},topsep=0pt,itemsep=0.5ex,partopsep=0ex]
        \item There are linear subspaces $L_k\subseteq \lin(T_{\pvpd}(\lambda^k))$ such that $L_k\to\lin(T_{\pvpd}(\bar\lambda))$,
        \item For all $p\in L_k$ with $\|p\|=1$, we have $\dist(0,T^{2}_\pvpd(\lambda^k,p))<M$.
    \end{enumerate}
    Then, it holds that:
    \begin{align}
        \label{thm3-3eq}
        {\eliminf}_{k\to\infty}\,\rd^2\vp(x^k|\revise{\D F(x^k)^\top}\lambda^k)\geq \langle \cdot, H(\bar x,\bar \lambda)\cdot\rangle+\iota_{\aff(\revise{\D F(\bar x)^{-1}}N_{\pvpd}(\bar\lambda))}.
    \end{align}
    In particular, if \ref{B1} and \ref{B2} hold, then condition \cref{cond_sosub1} is satisfied.
\end{thm}
\begin{proof}
    Using \cite[Corollary 4.3]{benko2022second} and \cite[Proposition 13.20]{rockafellar2009variational}, we have $\rd^2\vp(x^k|v^k)\geq \langle \cdot,H(x^k,\lambda^k)\cdot\rangle$ and applying \cite[Proposition 3.6]{ouyang2023partI}, it follows $\lambda^k\to\bar\lambda$. Thus, taking the epi-liminf and noticing that $\langle \cdot,H(x^k,\lambda^k)\cdot\rangle$ converges continuously, we obtain
    \begin{align}
        \label{thm3-3eq1}
        {\eliminf}_{k\to\infty}\,\rd^2\vp(x^k|v^k)\geq \langle \cdot, H(\bar x,\bar \lambda)\cdot\rangle.  
    \end{align}
    Applying \cite[Lemma 3.22]{ouyang2023partI}, we know $\aff(\revise{\D F(\bar x)^{-1}}N_{\pvpd}(\bar\lambda))=\revise{\D F(\bar x)^{-1}}\aff(N_{\pvpd}(\bar\lambda))$. Let us now take $w\notin \revise{\D F(\bar x)^{-1}}(\aff(N_{\pvpd}(\bar\lambda)))$, which implies $\|\Pi_{\lin(T_{\pvpd}(\bar\lambda))}(\revise{\D F(\bar x)}w)\|=2\epsilon$ for some $\epsilon>0$. By the definition of the epi-liminf, we may find $w^k\to w$ such that
    \[   ({\eliminf}_{k\to\infty}\,\rd^2\vp(x^k|v^k))(w)= {\liminf}_{k\to\infty} \, \rd^2\vp(x^k|v^k)(w^k).       \]
    Notice that the set convergence $L_k \to \lin(T_{\pvpd}(\bar\lambda))$ implies $\Pi_{L_k} \to \Pi_{\lin(T_{\pvpd}(\bar\lambda)}$. Hence, by continuity, we may assume $ \|\Pi_{L_k}(\revise{\D F(x^k)}w^k)\|>\epsilon$ for all $k \in \mathbb N$. Next, let us set $p^k={\Pi_{L_k}(\revise{\D F(x^k)}w^k)}/{\|\Pi_{L_k}(\revise{\D F(x^k)}w^k)\|}$. Then, due to $p^k\in L_k\subseteq \lin(T_{\pvpd}(\lambda^k))$ and $F(x^k)\in \partial\sigma_{\pvpd}^*(\lambda^k) = N_{\pvpd}(\lambda^k)$, it holds that $\langle p^k,F(x^k) \rangle = 0$ and $\langle p^k, \revise{\D F(x^k)}w^k\rangle=\|\Pi_{L_k}(\revise{\D F(x^k)}w^k) \|>\epsilon$. Applying \cref{lemma5-10} to $x^k$, $\lambda^k$, $w^k$ and $p^k$, we have 
    \[    \rd^2\vp(x^k|\revise{\D F(x^k)^\top}\lambda^k)(w^k)\geq  \frac{ \|\Pi_{L_k}(\revise{\D F(x^k)}w^k)\|^2 }{ 2M\|F(x^k)\|}+\langle w^k,H(x^k,\lambda^k)w^k\rangle                                    \]
    for all $k$. Consequently, taking the limit and using $F(x^k) \to F(\bar x) = 0$, we obtain:
    \begin{align}
        \label{thm3-3eq2}
        ({\eliminf}_{k\to\infty}\,\rd^2\vp(x^k|v^k))(w)=\infty.   
    \end{align}
    Therefore, \cref{thm3-3eq} follows from \cref{thm3-3eq1} and \cref{thm3-3eq2}. Under assumption \ref{B1}, it holds that $\partial \vp(x^k) = \revise{\D F(x^k)^\top}\partial \sigma_{\pvpd}(F(x^k))$ (for all $k$ sufficiently large), cf. \cite[Proposition 3.5]{ouyang2023partI}. Thus, the required properties of $\{\lambda^k\}_k$ and the conditions (i) and (ii) are ensured under \ref{B1} and \ref{B2}. Condition \cref{cond_sosub1} then follows from \cref{thm3-3eq}.
\end{proof}


\subsection{Verifying \texorpdfstring{\cref{cond_sosub2}}{(2.8)}} \label{sec:sosub2}
As illustrated in \cref{thm3-3}, the geometric requirements \ref{B2} on $\pvpd$ allowed us to establish an appropriate lower bound for the second subderivative of $\vp$. In order to verify \cref{cond_sosub2}, the generalized quadratic structure of the second subderivative---as stated in \cref{prop2-10}---will play a more important role. Since the function $\vp$ is often $C^2$-strictly decomposable everywhere \cite{milzarek2016numerical,ouyang2023partI}, we plan to use \cite[Lemma 5.3.27]{milzarek2016numerical} and \cite[Theorem 3.11]{ouyang2023partI} to connect the generalized quadratic structure of $\rd^2\vp$ to the strict complementarity condition. 

Moreover, to verify \cref{cond_sosub2}, we only need to find a specific sequence $\{(x^k,v^k)\}_k$ rather than controlling all possible sequences as in \cref{cond_sosub1}. 
\revise{As the normal cone $N_{\pvpd}(\bar\lambda)$ and the smallest exposed face $G_{\mathrm{ex}}({\bar\lambda})$ of $\pvpd$ containing $\bar\lambda$ will be utilized frequently in this subsection, we introduce the notations $\cN:=N_{\pvpd}(\bar\lambda)$ and $\cG:=G_{\mathrm{ex}}({\bar\lambda})$.}

\begin{assum} \revise{In addition to \ref{B1},} we consider \revise{the condition:
    \begin{enumerate}[label=\textup{\textrm{(B.\arabic*)}},topsep=0pt,itemsep=0.5ex,partopsep=0ex,start=3]  
            \item \label{B3}  For every pair $(x,\lambda)$ in a neighborhood of $(\bar x,\bar\lambda)$ satisfying $F(x)\in \ri(\cN)$ and  $\lambda\in \ri(\cG)$, the proximity operator $\proxs$ is differentiable at $ x+\tau \D F(x)^\top\lambda$. 
        \end{enumerate}
        }
    \end{assum}
    \begin{rem}
    \label{remark_B3}
     \revise{According to \cref{lem:exposed}, for every $x$ sufficiently close to $\bar x$ with $F(x)\in \ri(\cN)$, it holds that $\partial\sigma_\pvpd(F(x))=\cG$. Combining \cite[Proposition 3.5 (i)]{ouyang2023partI} with \cite[Theorem 6.6]{rockafellar1970convex}, the condition $\lambda\in \ri(\cG)$ implies $\D F(x)^\top\lambda \in \ri(\partial \vp(x))$. Therefore, }assumption \ref{B3} holds if $\vp$ is \revise{$C^2$-fully decomposable or $C^2$-partly smooth} \revise{at those $x$ in a neighborhood of $\bar x$ with $F(x)\in \ri(\cN)$}, cf. \cite[Theorem 3.11]{ouyang2023partI} \revise{and \cite{DanHarMal06,HuTiaPanWen23}}.
    \end{rem} 

 By \cite[Proposition 3.4(ii)]{ouyang2023partI}, we know that the strict constraint qualification $\revise{\D F(\bar x)}\R^n-\cN=\R^m$ (see \cref{eq:cq-strict}) is stable on a neighborhood of $\bar x$.
\begin{prop}
    \label{prop3-4}
   Let condition \ref{B1} hold. Then there is a neighborhood $U$ of $\bar x$ such that for all $x\in U$ it holds that $\D F(x)\R^n-\cN=\R^m$. 
\end{prop}

To verify \cref{cond_sosub2}, we need to select a proper sequence $(x^k,v^k)\to(\bar x,\bar v)$ to satisfy the conditions listed in \cref{prop2-10}. We start with the choice of $\{x^k\}_k$.
\begin{prop}
    \label{nprop3-5}
    Let \ref{B1} hold. There exists $\{x^k\}_k$ with $x^k\to\bar x$, $F(x^k)\in \ri(\cN)$, $\partial\vp(x^k)=\revise{\D F(x^k)^\top}\cG$ (for all $k$), and $\revise{\D F(x^k)^{-1}}\aff(\cN)\to \revise{\D F(\bar x)^{-1}}\aff(\cN)$.
\end{prop}
\begin{proof} This follows from \cite[Proposition 3.9]{ouyang2023partI}. The representation $\partial \vp(x^k)=\revise{\D F(x^k)^\top}\cG$ is due to \cite[Proposition 3.5 (i)]{ouyang2023partI} and \cref{lem:exposed} (ii).
\end{proof}

Our next task is to give an exact expression of $\rd^2\vp(x^k|\revise{\D F(x^k)^\top}\lambda^k)$ (on the linear subspace $\revise{\D F(x^k)^{-1}}\aff(\cN)$) to calculate the epi-limit.

\begin{prop}
    \label{prop2-30}
    Let \ref{B1} be satisfied. Then, there exists a neighborhood $U$ of $\bar x$ such that for all $x\in U$ with $F(x)\in\ri(\cN)$ and $\lambda\in\ri(\cG)$, it holds that 
    \[  \rd^2\vp(x|\revise{\D F(x)^\top}\lambda)(w)=\langle w, H(x,\lambda)w \rangle, \quad \forall~w\in \revise{\D F(x)^{-1}}\aff(\cN).  \]
\end{prop}
\begin{proof}
    By \cref{prop3-4}, there is a neighborhood $U$ of $\bar x$ such that $\revise{\D F(x)}\R^n-\cN=\R^m$ holds for all $x\in U$. By \cite[Corollary 4.3]{benko2022second}, we again have $\rd^2\vp(x|\revise{\D F(x)^\top}\lambda)\geq \langle \cdot,H(x,\lambda)\cdot\rangle$. Hence, it suffices to prove the converse inequality. Now assume $w\in \revise{\D F(x)^{-1}}\aff(\cN)$ and define $K:=\{x\in\R^n:F(x)\in \cN\}$. Applying \cite[Corollary 2.91]{bonnans2013perturbation} and $F(x) \in \ri(\cN)$, it follows $T_K(x) = \revise{\D F(x)^{-1}}T_{\mathcal N}(F(x)) = \revise{\D F(x)^{-1}}\aff(\cN)$, see \cite[Section A.5.3]{HirLem01}. Thus, we can infer $w \in T_K(x)$ and due to $\lambda \in \ri(\cG)$ and \cref{lem:exposed} (i), it holds that $N_{\pvpd}(\lambda)=\cN$. Hence, there exist $t_k\downarrow 0$ and $w^k\to w$ such that $F(x+t_kw^k)\in\cN=N_{\pvpd}(\lambda)$.
    %
By the definition of the second subderivative, this yields:
    \begin{align*}
        &\hspace{-6ex}\rd^2\vp(x|\revise{\D F(x)^\top}\lambda)(w)\leq \lim_{k\to\infty}
        \Delta_{t_k}^2  \;\! (\sigma_{\pvpd}\circ F)(x|\revise{\D F(x)^\top}\lambda)(w^k) \\
        &=\lim_{k\to\infty}\frac{\langle \lambda,F(x+t_kw^k)-F(x)-t_k\revise{\D F(x)}w^k\rangle}{\frac{1}{2}t_k^2}=\langle w,H(x,\lambda),w\rangle.
    \end{align*}      
\end{proof}

We now combine the different propositions to show one of our main results. 

\begin{thm}
    \label{thm2-32}
    \revise{Assume \ref{B1}--\ref{B3}}. Then, condition \cref{cond_sosub2} is satisfied.  
\end{thm}
\begin{proof}
   \revise{ Let $\{x^k\}_k$ be given as in \cref{nprop3-5}, and select $\ri(\cG)\ni \lambda^k\to \bar\lambda$}. 
   \revise{Then, by \ref{B3} and \cite[Proposition 2.2 and 2.3]{ouyang2023partI}}, $\proxs$ is differentiable at $z^k=x^k+\tau v^k$, \revise{$v^k = \D F(x^k)^\top\lambda^k$} with $0\preceq \revise{\D\proxs(z^k)}\preceq\frac{1}{1-\tau\rho} I$. Taking a subsequence if necessary, we may assume $\revise{\D\proxs(z^k)}\to R$ for some $0\preceq R\preceq\frac{1}{1-\tau\rho} I$. Applying \cite[Proposition 2.2]{ouyang2023partI}, we know that $\vp$ is twice epi-differentiable at $x^k$ for $v^k$, $\rd^2\vp(x^k|v^k)$ is generalized quadratic and converges epigraphically. Noticing $T_{\pvpd}(\lambda^k) = T_{\pvpd}(\bar\lambda)$ \revise{and using \ref{B2} and \cref{thm3-3},} we have
    \begin{align}
        \label{thm3-8eq1}
      {\eliminf}_{k\to\infty}\,\rd^2\vp(x^k|v^k)\geq \langle \cdot, H(\bar x,\bar \lambda)\cdot\rangle+\iota_{\aff(\revise{\D F(\bar x)^{-1}}\cN)}.         
    \end{align}
    As before, by \cite[Lemma 3.22]{ouyang2023partI}, it holds that $\aff(\revise{\D F(\bar x)^{-1}}\cN) = \revise{\D F(\bar x)^{-1}}\aff(\cN)$.
    Let $w\in\revise{\D F(\bar x)^{-1}}\aff(\cN)$ be arbitrary. By \cref{nprop3-5}, we can choose $w^k\to w$ with $w^k\in \revise{\D F(x^k)^{-1}}\aff(\cN)$. Applying \cref{prop2-30}, we obtain
    \begin{align}
        \label{thm3-8eq2}
        ({\eliminf}_{k\to\infty}\,\rd^2\vp(x^k|v^k))(w)\leq{\lim}_{k\to\infty}\,\rd^2\vp(x^k|v^k)(w^k)=\langle w,H(\bar x,\bar\lambda)w\rangle.  
    \end{align}
     Condition \cref{cond_sosub2} then follows from \cref{thm3-8eq1} and \cref{thm3-8eq2} and by noticing that the sequence $\{\rd^2\vp(x^k|v^k)\}_k$ epi-converges. 
\end{proof}

\revise{In \cref{fig:new}, we provide an overview of our results and of the different algebraic, variational, and geometric conditions and their connections to the $\SSOSC$ and \cref{cond_gj}.}

\begin{figure}[t]
\tikzstyle{block} = [rectangle, draw, fill=gray!3, minimum width=11cm, minimum height=3em]
\tikzstyle{line} = [draw, thick, -latex']

\centering
{\small
\begin{tikzpicture}
\draw [draw, very thick, -latex'] (1,0)--(1,-1.5);
\draw [draw, very thick, -latex'] (-1,-2)--(-1,-0.5);
\node [block] (init) at (0,0) {$\langle h,[\nabla^2f(\bar x)+Q]h\rangle\geq \sigma\|h\|^2 \quad \forall~h\in \aff(S)$};
\node at (-4.5,0) {SSOSC:};
\node [block] (identify) at (0,-2) {$D\nabla^2f(\bar x)D+{\tau}^{-1} D(I-D)\succeq \sigma D^2 \quad \forall~D\in\partial\proxs(\bar z)$};
\node[right] at (1.2,-0.8) {Properties \cref{cond1_gj} / \cref{weakcond1_gj}:}; 
\node[right] at (1.2,-1.2) {\cref{ssoscimcondgj} / \cref{prop3-5}};
\node[left] at (-1.2,-0.8) {Property \cref{cond2_gj}:}; 
\node[left] at (-1.2,-1.2) {\cref{condgjimssosc}};
\draw [draw=black] (-6.4,-3) rectangle (6.4,0.75);
\node[right] at (2,-2.75) {{\footnotesize Base assumptions: \ref{A1}--\ref{A2}}};
\draw [draw=white,fill=gray!3] (-6.4,-4.25) rectangle (-3.0,-3.25);
\draw [draw=black] (-6.4,-5.55) rectangle (-3.0,-3.25);
\draw [draw=black] (-6.4,-4.25) -- (-3.0,-4.25);
\draw [draw=black,densely dotted] (-6.4,-4.9)-- (-3.0,-4.9);
\draw [draw=white,fill=gray!3] (-1.7,-4.55) rectangle (1.7,-3.25);
\draw [draw=black] (-1.55,-5.55) rectangle (1.55,-3.25);
\draw [draw=black] (-1.55,-4.25) -- (1.55,-4.25);
\draw [draw=black,densely dotted] (-1.55,-4.9)-- (1.55,-4.9);
\draw [draw=white,fill=gray!3] (3.0,-5.55) rectangle (6.4,-3.25);
\draw [draw=black] (3.0,-5.55) rectangle (6.4,-3.25);
\draw [draw=black] (3.0,-4.25) -- (6.4,-4.25);
\draw [draw=black,densely dotted] (3.0,-4.9)-- (6.4,-4.9);
\node at (-4.7,-3.6) {\footnotesize Geometric Conditions};
\node at (-4.7,-3.9) {\footnotesize Decomposability};
\node at (-4.7,-4.6) {\ref{B1}, \ref{B2}};
\node at (-4.7,-5.2) {\ref{B1}, \ref{B2}, \ref{B3}};
\node at (0,-3.6) {\footnotesize Variational};
\node at (0,-3.9) {\footnotesize Conditions};
\node at (0,-4.6) {\cref{cond_sosub1}};
\node at (0,-5.2) {\cref{cond_sosub2}};
\node at (4.7,-3.6) {\footnotesize Algebraic};
\node at (4.7,-3.9) {\footnotesize Conditions};
\node at (4.7,-4.6) {\cref{weakcond1_gj}};
\node at (4.7,-5.2) {\cref{cond2_gj}};
\draw [draw, very thick, -latex'] (-3,-4.6)--(-1.55,-4.6);
\draw [draw, very thick, -latex'] (-3,-5.2)--(-1.55,-5.2);
\draw [draw, very thick, -latex'] (1.55,-4.6)--(3,-4.6);
\draw [draw, very thick, -latex'] (1.55,-5.2)--(3,-5.2);
\node [above] at (-2.275,-4.5) {\scriptsize \textcolor{blue}{Thm.}~\ref{thm3-3}};
\node [below] at (-2.275,-5.3) {\scriptsize \textcolor{blue}{Thm.}~\ref{thm2-32}};
\node [above] at (2.275,-4.5) {\scriptsize \textcolor{blue}{Prop.}~\ref{proposition:v1-p3}};
\node [below] at (2.275,-5.3) {\scriptsize \textcolor{blue}{Prop.}~\ref{prop2-10}};
\end{tikzpicture}
}
\caption{\revise{Diagram illustrating the relationships between the assumptions and conclusions in this work. The algebraic conditions \cref{cond1_gj} / \cref{weakcond1_gj} and \cref{cond2_gj} enable us to establish the equivalence between the $\SSOSC$ and second-order condition \cref{cond_gj}. These two properties are implied by the variational conditions \cref{cond_sosub1} and \cref{cond_sosub2}, which can be further derived using the geometric assumptions \ref{B1}, \ref{B2}, and \ref{B3} that leverage the decomposable structure of the function $\vp$.} \label{fig:new}}
\end{figure}

\section{Equivalence between Other Conditions}
\label{sec:eq_ssosc}
In this section, we investigate the equivalence between the $\SSOSC$ and other variational concepts for \cref{prob1}. Specifically, we show that the strong second-order sufficient condition, the strong metric regularity of the subdifferential, the normal-map, the natural residual, and the uniform invertibility of the generalized derivatives in $\mathcal M_{\mathrm{nat}}^\tau(\bar x)$ and $\mathcal M_{\mathrm{nor}}^\tau(\bar z)$ are all equivalent under the setting studied in \Cref{sec:eq_ssonc}. Our aim is to prove the following theorem.
\begin{thm}
    \label{thm4-1}
    Let the assumptions \ref{B1}--\ref{B3} hold. Then all of the following statements are equivalent:
\begin{enumerate}[label=\textup{\textrm{(\roman*)}},topsep=0pt,itemsep=0ex,partopsep=0ex]
        \item There is a constant $\sigma_1>0$ such that
        \[ \langle  h, (\nabla^2f(\bar x)+Q) h \rangle \geq \sigma_1\|h\|^2 \quad \forall~h\in \aff(S) = \affS. \] 
        \item There are neighborhoods $U$ of $\bar x$ and $V$ of $0$ such that for every $v\in V$ there is a point $u\in U\cap (\partial\psi)^{-1}(v)$ such that:
        \[\forall~x\in U,\quad  \psi(x)\geq \psi(u)+\iprod{v}{x-u}+\frac{\sigma_2}{2}\|x-u\|^2.
            \]     
        \item The subdifferential $\partial \psi$ is strongly metrically regular at $\bar x$ for $0$ and the necessary condition $\mathrm{d}^2\psi(\bar x|0)(h) \geq 0$ is satisfied for all $h \in \Rn$. 
        \item The normal map $\oFnor$ is strongly metrically regular at $\bar z$ for $0$ and the second-order necessary condition $\mathrm{d}^2\psi(\bar x|0)(h) \geq 0$ holds for all $h \in \Rn$. 
        \item The second-order necessary condition $\mathrm{d}^2\psi(\bar x|0)(h) \geq 0$ holds for all $h \in \Rn$ and every matrix $M \in \mathcal M_{\mathrm{nor}}^\tau(\bar z)$ is invertible.
        \item The strong second-order necessary condition $\langle h,(\nabla^2f(\bar x)+Q)h\rangle \geq 0$ holds for all $h \in \affS$ and $M=\nabla^2f(\bar x)D+\frac1\tau (I-D)$ is invertible for all $D\in\partial_B \proxs(\bar z)$.
        \item There is a constant $\sigma_3>0$ such that:
        \[  \forall~D\in\partial_B\proxs(\bar z), \quad D\nabla^2f(\bar x)D+\tau^{-1} D(I-D)\succeq \sigma_3D^2. \]
        \item There is a constant $\sigma_4>0$ such that:
        \[  \forall~D\in\partial\proxs(\bar z), \quad D\nabla^2f(\bar x)D+\tau^{-1} D(I-D)\succeq \sigma_4D^2. \]
        \item There is a neighborhood $U$ of $\bar z$ and $\sigma_5>0$ such that:
        \[  \forall~z\in U, \, \forall~D\in\partial\proxs(z), \quad D\nabla^2f(\proxs(z))D+\tau^{-1} D(I-D)\succeq \sigma_5D^2.                         \]
    \end{enumerate}
    Each of the conditions \textup{\textrm{(i)}}--\textup{\textrm{(ix)}} implies the following statement---which is further equivalent to \textup{\textrm{(i)}}--\textup{\textrm{(ix)}} provided that $I-\tau\nabla^2f(\bar x)$ is invertible:
    \begin{itemize}
        \item[\textup{\textrm{(\rnum{10})}}] The natural residual $\oFnat$ is strongly metrically regular at $\bar x$ for $0$ and the necessary condition $\mathrm{d}^2\psi(\bar x|0)(h) \geq 0$ is satisfied for all $h \in \Rn$. 
    \end{itemize}
\end{thm}

Condition (i) is the $\SSOSC$ \cref{ssosc}. Condition (ii) is an equivalent characterization of tilt-stability of the local minimizer $\bar x$. It is further equivalent to a strong second-order condition based on Mordukhovich's coderivative, see \cite[Theorem 4.6 (ii)]{DruMorNhg14}. We refer to \cite{PolRoc98,DruLew13,DruMorNhg14} for more background and additional relations. Conditions (iii), (iv), and (x) connect the $\SSOSC$ to the strong metric regularity of the subdifferential, the normal map, and the natural residual. Following the discussion after \cref{cond_gj}, invertibility of the matrices $M \in \mathcal M_{\mathrm{nor}}^\tau(\bar z)$ in condition (v) is equivalent to the CD-regularity of $\oFnor$ at $\bar z$. Due to $\mathcal M_{\mathrm{nat}}^\tau(\bar x) = \tau \mathcal M_{\mathrm{nor}}^\tau(\bar z)^\top$, condition (v) can also be equivalently expressed using $\mathcal M_{\mathrm{nat}}^\tau(\bar x)$ and the natural residual $\oFnat$. In particular, 
noticing $\partial\oFnat(\bar x)h \subseteq \mathcal M_{\mathrm{nat}}^\tau(\bar x)h$, $h \in \Rn$, (see \cite{clarke1990optimization}), condition (v) implies CD-regularity of $\oFnat$ at $\bar x$. (If $I-\tau\nabla^2f(\bar x)$ is invertible, then it follows $\mathcal M_{\mathrm{nat}}^\tau(\bar x) = \partial\oFnat(\bar x)$ and (v) is fully equivalent to CD-regularity of $\oFnat$, cf. \cite[Lemma 1]{chan2008constraint}). Condition (vi) is a BD-regularity-type analogue of (v).
Conditions (vii) and (viii) are the Bouligand and Clarke variant of the second-order condition \cref{cond_gj} that was initially proposed in \cite{ouyang2021trust}. The generalized version (ix) of (viii) was used in \cite{ouyang2021trust} to prove global-to-local transition and superlinear convergence of a trust region-based semismooth Newton method. 

We note that when $\psi$ is convex, then $\rd^2\psi(\bar x|0)$ is convex due to \cite[Theorem 13.20]{rockafellar2009variational}, and the strong second-order necessary condition holds automatically. Therefore, in the convex case, we obtain  equivalence between BD- and CD-regularity. In \cite{chan2008constraint}, similar results have been established \revise{for a KKT-based stationary measure} for convex semidefinite programming problems. 

We continue with two counterexamples before starting with the proof of \cref{thm4-1}. First, one might wonder whether the strong second-order necessary condition in (vi) can be relaxed to the regular second-order necessary condition. The following counterexample provides a negative answer. 
\begin{example}
    \label{exam4-2}
   We set $f(x) :=\frac{1}{2}\iprod{x}{Ax}$ and $\vp(x) :=\iota_{\R^2_+}(x)$, where $a_{11} = a_{12} = a_{21} = 2$ and $a_{22} = 1$.
    As \revise{a polyhedral function}, $\vp$ is clearly $C^2$-fully decomposable at all $x \in \R_+^2$, see \cite{shapiro2003class,milzarek2016numerical}. Defining $\bar x :=(0,0)^\top$, we have $\rd^2\psi(\bar x|0)(w)=\iprod{w}{Aw}+\iota_{\R_+^2}(w)$ for all $w$.     
    It is not hard to see that the second-order necessary condition is satisfied at $\bar x$. We may further set $\tau = 1$, $\bar z=\bar x-\tau\nabla f(\bar x) = \bar x$, and $\bar\lambda=0$. Elementary calculus shows $\partial_B\Pi_{\R_+^2}(\bar z)=\{I,B_1,B_2,0\}$, where $B_1= \mathrm{diag}(1,0)$ and $B_2 = \mathrm{diag}(0,1)$.
    The set $\{AD+I-D:D\in\partial_B\Pi_{\R_+^2}(\bar z)\}$ contains the following four matrices:
    \begin{align*}
       A,\quad  \begin{pmatrix}
            2 & 0 \\
            2 & 1
        \end{pmatrix},\quad  \begin{pmatrix}
            1 & 2 \\
            0 & 1
        \end{pmatrix},\quad I.
    \end{align*}
    They are all invertible. But the CD-regularity-type condition \textup{\textrm{(v)}} fails to hold, since $\aff(\R^2_+) = \R^2$ and $A$ is not positive definite.
\end{example} 

Next, we present an example to demonstrate that the additional invertibility of $I-\tau\nabla^2f(\bar x)$ is necessary to establish equivalence between (\rnum{10}) and the other conditions.
\begin{example}
    We consider $f(x) := \frac12\iprod{x}{Ax}$ and $\vp(x) := \iota_K(x)$, where $A := \mathrm{diag}(2,-1)$ and $K :=\{x \in \R^2: x_1\geq |x_2|\}$.
    As in \cref{exam4-2}, $\vp$ is  $C^2$-fully decomposable at all $x \in K$ and for $\bar x = (0,0)^\top$ we obtain $\rd^2\psi(\bar x|0)(w)=\iprod{w}{Aw}+\iota_{K}(w)$ for all $w$.
   As before, the second-order necessary condition $\rd^2\psi(\bar x|0)(h) \geq 0$ is satisfied at $\bar x$ for all $h \in \R^2$. Moreover,  the projection onto $K$ can be computed via:  
    \begin{align*}  
    \begin{array}{llll}
    \Pi_{K}(x)= \frac{x_1+x_2}{2}(1,1)^\top & \text{if } x_2\geq |x_1|,  & \Pi_{K}(x)= \frac{x_1-x_2}{2}(1,-1)^\top & \text{if } x_2\leq -|x_1|, \\[1ex]
    \Pi_{K}(x)= (0,0)^\top & \text{if } x_1\leq -|x_2|,  & \Pi_{K}(x)= x & \text{if } x_1\geq |x_2|.
    \end{array}
    \end{align*}
    We now set $\tau=\frac{1}{2}$. Then, we obtain $x-\tau \nabla f(x)= x - \frac12 Ax = (0,\frac{3}{2}x_2)^\top$ and
    \begin{align*}
        \oFnat(x)=x-\Pi_K(x-\tau\nabla f(x))= \begin{pmatrix} x_1- \frac34 |x_2| & \frac14 x_2 \end{pmatrix}^\top.
    \end{align*} 
    To show that $\oFnat$ is strongly metrically regular at $\bar x$ for $0$, we calculate its inverse, i.e., we consider the nonlinear equation $\oFnat(x)=y$. It readily follows
    %
     \begin{align*}
        (\oFnat)^{-1}(y) = \begin{pmatrix} y_1 + 3|y_2| & 4y_2 \end{pmatrix}^\top.
    \end{align*}
    Therefore, $(\oFnat)^{-1}$ is single-valued and globally Lipschitz continuous (as it is piecewise linear). Consequently, by \cite[Theorem 9.43]{rockafellar2009variational}, $\oFnat$ is strongly metrically regular at $\bar x$ for $0$. However, condition \textup{(\rnum{1})} in \cref{thm4-1} does not hold, since $A$ is not positive semidefinite and $\aff(K)=\R^2$.
\end{example} 

Next, we state a preparatory result that is used to prove the implication ``(viii)\,\!$\implies$\,\!(ix)''. The proof of \cref{prop3-7} applies similar techniques as in \cite[Proposition 7.11]{ouyang2021trust} and can be found in \cref{app:invert}.
\begin{prop}
    \label{prop3-7}
    Let $D\in \bS^n_+$, $B\in \bS^n$ and $\tau, \sigma > 0$ be given with:
    \[    DBD+{\tau^{-1}}D(I-D)\succeq \sigma D^2.    \]
    Then, there are neighborhoods $U$ of $D$ and $V$ of $B$ such that for every $\tilde D\in U\cap \bS^n_+$ and $\tilde B\in V\cap \bS^n$, we have $\tilde D\tilde B\tilde D+{\tau^{-1}}\tilde D(I-\tilde D)\succeq \tfrac{\sigma}{4}\tilde D^2$.                            
\end{prop}
 
We now present the proof of \cref{thm4-1}.
 
\begin{proof}[Proof of \cref{thm4-1}]
    Note that the equivalences (i)\,\!$\iff$\,\!(vii)\,\!$\iff$\,\!(viii) have been established in \Cref{sec:ssonc} and \Cref{sec:eq_ssonc}. In fact, we immediately have (i)\,\!$\iff$\,\!(viii) and (viii)\,\!$\implies$\,\!(vii). To show the remaining implication (vii)\,\!$\implies$\,\!(i), we now argue that the matrix $P = \Pi_{\affS} (I+\tau Q)^{-1}\Pi_{\affS}$ in \cref{cond2_gj} is actually an element of $\partial_B\proxs(\bar z)$ under \cref{cond1_gj}. By \cref{cond1_gj} and due to $\mathcal R(D) \subseteq \affS$, it holds that
    \[ D=\Pi_{\mathcal R(D)}(I+\tau A)^{-1}\Pi_{\mathcal R(D)} \preceq D=\Pi_{\mathcal R(D)}(I+\tau \Pi_{\mathcal R(D)}Q\Pi_{\mathcal R(D)})^{-1}\Pi_{\mathcal R(D)} \preceq P \]
    for all $D \in \partial\proxs(\bar z)$. We claim that $P$ is an extreme point of the set $\partial\proxs(\bar z) = \conv(\partial_B\proxs(\bar z))$. Let us suppose that this claim is wrong, i.e., there exist $P_1, P_2 \in \partial\proxs(\bar z)$, $P_1 \neq P_2$, and $\mu \in (0,1)$ such that $P = \mu P_1 + (1-\mu) P_2$. Then, it follows
    \[ 0 \preceq \mu(P - P_1) = (1-\mu)(P_2 - P) \preceq 0 \]
    and we can infer $P-P_1 = P_2 - P = 0$ which yields a contradiction. Hence, $P$ is an extreme point of $\partial\proxs(\bar z)$ and we have $P \in \partial_B\proxs(\bar z)$. The strategies utilized in the proof of \cref{ssoscimcondgj} are then directly applicable to complete the implication (vii)\,\!$\implies$\,\!(i). 
    We continue the proof of \cref{thm4-1} in a step-by-step fashion:

    \textbf{(\rnum{9})\,\!$\iff$\,\!(\rnum{8})}: Clearly, we have (\rnum{9})\,\!$\implies$\,\!(\rnum{8}). The direction (\rnum{8})\,\!$\implies$\,\!(\rnum{9}) follows from \cref{prop3-7}, the local Lipschitz continuity of $\proxs$, and the upper semicontinuity of $\partial\proxs$. 

      \textbf{(\rnum{8})\,\!$\implies$\,\!(\rnum{5})}: The second-order necessary condition in (\rnum{5}) follows from (\rnum{8})\,\!$\iff$\,\!(\rnum{1}). Invertibility of the matrices $M \in \mathcal M_{\mathrm{nor}}^\tau(\bar z)$ was verified after introducing \cref{cond_gj}. 

     \textbf{(\rnum{5})\,\!$\implies$\,\!(\rnum{4})}:  As discussed, the invertibility condition in (\rnum{5}) implies CD-regularity of $\oFnor$ at $\bar z$. Thus, (\rnum{4}) follows from the classical Clarke inverse function theorem, cf. \cite[Theorem 7.1.1]{clarke1990optimization} or \cite[Theorem 4D.3]{DonRoc14}. 

    \textbf{(\rnum{4})\,\!$\implies$\,\!(\rnum{3})}: By assumption, there exists a neighborhood $U$ of $\bar z$ and $\kappa > 0$ such that $\|\Fnors(z_1)-\Fnors(z_2)\|\geq \kappa\|z_1-z_2\|$ for all $z_1, z_2 \in U$. In addition, there is a neighborhood $V$ of $0$ such that $(\oFnor)^{-1} : V \to \Rm$ is well-defined, single-valued, and Lipschitz continuous, cf. \cite[Proposition 3G.1]{DonRoc14} and 
    we may assume $(\oFnor)^{-1}(V) \subseteq U$. 
     %
    %
    %
    Let $(x_1,y_1), (x_2,y_2) \in (\Rn \times V) \cap \gph(\partial \psi)$ be arbitrary. Then, we can write:
    \[      y_1=\nabla f(x_1)+v_1, \; y_2=\nabla f(x_2)+v_2 \quad \text{for some} \quad v_1\in\partial\vp(x_1) ,\; v_2\in\partial \vp(x_2).                   \]
    Thus, setting $z_1=x_1+\tau v_1$, $z_2=x_2+\tau v_2$, we have $x_1=\proxs(z_1)$, $x_2=\proxs(z_2)$ and  $y_1 = \oFnor(z_1)$,  $y_2 = \oFnor(z_2)$. Utilizing $y_1, y_2 \in V$, this implies $z_1 = (\oFnor)^{-1}(y_1)$ and $z_2 = (\oFnor)^{-1}(y_2)$ and we can infer $z_1, z_2 \in U$. Moreover, this verifies that $(\partial\psi)^{-1}$ is single-valued on $V$, i.e., we have $(\partial\psi)^{-1}(y) = \{\proxs((\oFnor)^{-1}(y))\}$ for all $y \in V$. 
    In addition, applying the Lipschitz continuity of the proximity operator, it follows
    %
    \[ \kappa_p \kappa\|x_1-x_2\| \leq \kappa \|z_1-z_2\| \leq \|\Fnors(z_1)-\Fnors(z_2)\| = \|y_1-y_2\|, \]
    for some $\kappa_p > 0$. This establishes the strong metric regularity of $\partial\psi$ at $\bar x$ for $0$.

    \textbf{(\rnum{3})\,\!$\implies$\,\!(\rnum{2})}: The necessary condition $\mathrm{d}^2\psi(\bar x|0)(h) \geq 0$, $h \in \Rn$, implies condition (3.6) in \cite[Corollary 3.3(\rnum{1})]{DruMorNhg14}; see, e.g., \cite[Theorem 13.24]{rockafellar2009variational}. Therefore, by \cite[Corollary 3.3]{DruMorNhg14}, $\bar x$ is a local minimum of $\psi$. The conclusion then follows from \cite[Theorem 3.7]{DruMorNhg14}.

    \textbf{(\rnum{2})\,\!$\implies$\,\!(\rnum{1})}: Following the proof of \cref{thm2-32} there are sequences $x^k\to\bar x$ and $\lambda^k\to\bar\lambda$ such that:
    \[  \rd^2\vp(x^k|v^k)(w)=\langle w, H(x^k,\lambda^k)w \rangle \quad \forall~w\in \revise{\D F(x^k)^{-1}}\aff(N_{\pvpd}(\bar \lambda)), \]
    where $v^k=\revise{\D F(x^k)^\top}\lambda^k$. In particular, setting $u^k=\nabla f(x^k)+v^k$, this implies 
    \begin{align}
        \label{thm4-1eq2}
        \rd^2\psi(x^k|u^k)(w)=\langle w, [\nabla^2f(x^k)+H(x^k,\lambda^k)] w\rangle \quad \forall~w\in \revise{\D F(x^k)^{-1}}\aff(N_{\pvpd}(\bar\lambda)).
    \end{align}
    By \cite[Theorem 3.7]{DruMorNhg14}, $(\partial\psi)^{-1}$ is single-valued around $0$. Hence, using $u^k \to \nabla f(\bar x) + \bar v = 0$ and (\rnum{2}), it holds that $(\partial\psi)^{-1}(u^k) = \{x^k\}$ and
%
    \begin{align}
        \label{thm4-1eq1}
        \psi(x)\geq \psi(x^k)+\langle u^k,x-x^k\rangle+\frac{\sigma_2}{2}\|x-x^k\|^2  
    \end{align}
    for all $x$ in a neighborhood of $\bar x$ and all $k$. Let us fix $w\in \aff(S) = \revise{\D F(\bar x)^{-1}}\aff(N_{\pvpd}(\bar\lambda))$. By \cref{nprop3-5}, we can choose $w^k\to w$ with $w^k\in \revise{\D F(x^k)^{-1}}\aff(N_C(\bar\lambda))$ and by the definition of $\rd^2\psi(x^k|u^k)$, there are $w^{k,i}\to w^k$ and $t_i \downarrow 0$ such that:
    \[       \rd^2\psi(x^k|u^k)(w^k)=\lim_{i\to\infty}\frac{\psi(x^k+t_iw^{k,i})-\psi(x^k)-t_i\langle u^k,w^{k,i}\rangle}{\frac{1}{2}t_i^2}\geq \sigma_2\|w^k\|^2, \]
    where we applied \cref{thm4-1eq1}. Thus, using \cref{thm4-1eq2}, we can infer $\langle w^k, [\nabla^2f(x^k)+H(x^k,\lambda^k)] w^k\rangle \geq \sigma_2\|w^k\|^2$ for all $k$. 
    Taking the limit $k\to\infty$, this shows that the $\SSOSC$ in (i) holds.

    \textbf{(\rnum{5})\,\!$\implies$\,\!(\rnum{6})\,\!$\implies$\,\!(i)}: By the previous steps, we have (\rnum{5})\,\!$\iff$\,\!(\rnum{1}) and hence, (\rnum{5}) implies that the strong second-order necessary condition is satisfied. The BD-regularity like condition in (vi) is clearly implied by (v). The direction (vi)\,\!$\implies$\,\!(i) is similar to the proof of \cref{ssoscimcondgj} and ``(vii)\,\!$\implies$\,\!(i)'' by using \cref{cond2_gj} and \cite[Lemma 6.3]{ouyang2021trust}.

    \textbf{(\rnum{5})\,\!$\implies$\,\!(\rnum{10})}: Noticing the CD-regularity of $\oFnat$ at $\bar x$, the proof is identical to the direction (\rnum{5})\,\!$\implies$\,\!(\rnum{4}). 
    
     \textbf{(\rnum{10})\,\!$\implies$\,\!(\rnum{4})}: Suppose that (\rnum{4}) is not true. Then, we can find sequences $\{z^k_1\}_k$, $\{z^k_2\}_k$ with $z^k_1,z^k_2\to\bar z$, $z_1^k \neq z_2^k$, and $ \alpha_k\to 0$ such that:
      \be \label{eq:contra-fnor-1} \| \Fnors(z^k_1)-\Fnors(z^k_2)\|\leq  \alpha_k\|z^k_1-z^k_2\|. \ee
 %
    We set $G(x):=x-\tau \nabla f(x)$. Since $\revise{\D G(\bar x)} = I-\tau\nabla^2f(\bar x)$ is invertible, $G$ is a local diffeomorphism at $\bar x$ satisfying $G(\bar x)=\bar z$. In particular, both $G$ and $G^{-1}$ are well-defined around $\bar x$ and $\bar z$ and we can define $x_1^k := G^{-1}(z_1^k)$ and $x_2^k := G^{-1}(z_2^k)$. (Without loss of generality, we assume that $x_1^k$ and $x_2^k$ are well-defined in this way for all $k \in \mathbb N$). The continuity of $G^{-1}$ further implies $x_1^k, x_2^k \to \bar x$ as $k \to \infty$. Moreover, utilizing the strong metric regularity of $\oFnat$ at $\bar x$ for $0$ and the continuity of $\revise{\D G}$, we may assume
    \be \label{eq:contra-fnor-2}  \|\oFnat(x^k_1)-\oFnat(x^k_2)\|\geq \kappa\|x^k_1-x^k_2\|, \quad \|G(x^k_1) - G(x^k_2)\| \leq L \|x_1^k - x_2^k\|         \ee
    for all $k$ and for suitable $\kappa, L > 0$. In addition, for all $x, z \in \Rn$ with $G(x) = z$, we can represent the normal map and natural residual as follows: 
    \[ \oFnor(z) = \tau^{-1} [G(x) - G(\proxs(z))] \quad \text{and} \quad \oFnat(x) = x - \proxs(z). \]
    %
    Setting $p_1^k = \proxs(z_1^k)$ and $p^k_2 = \proxs(z^k_2)$, we then obtain
  %
    %
    \begin{align*}  \tau[\oFnor(z_1^k) -  \oFnor(z_2^k)] &= G(x_1^k)-G(x^k_2) - [G(\proxs(z^k_1)) - G(\proxs(z^k_2))] \\ & \hspace{-8ex} =  B(x_1^k,x^k_2)(x_1^k-x^k_2) - B(p^k_1,p^k_2)(p^k_1-p^k_2) \\ &\hspace{-8ex} = [B(x_1^k,x_2^k)-B(p^k_1,p^k_2))](x^k_1-x^k_2) + B(p_1^k,p^k_2)[\oFnat(x_1^k) - \oFnat(x^k_2)] \end{align*}
    where $B(y,y^\prime) =  \int_0^1 \revise{\D G(y^\prime+t (y-y^\prime))}\,\mathrm{d}t$. Noticing $\bar x = \proxs(\bar z)$, we further have $B(x_1^k,x_2^k), B(p_1^k,p_2^k) \to \revise{\D G(\bar x)} = I-\tau \nabla^2 f(\bar x)$ as $k \to \infty$. Hence, there are $\bar\sigma > 0$ and $K \in \mathbb N$ such that $\sigma_{\min}{(B(p^k_1,p^k_2))} \geq \bar \sigma$ and
    \[ \quad \| [B(x_1^k,x_2^k)-B(p^k_1,p^k_2))](x^k_1-x^k_2) \|\leq \frac{\kappa\bar\sigma}{2} \|x^k_1-x^k_2\| \]
    for all $k \geq K$ which by \cref{eq:contra-fnor-2} implies $\|\oFnor(z_1^k) -  \oFnor(z_2^k)\| \geq \frac{\kappa\bar\sigma}{2\tau} \|x_1^k-x_2^k\|$. Combining this estimate with \cref{eq:contra-fnor-1} and \cref{eq:contra-fnor-2}, we have 
    %
    $\frac{\kappa\bar\sigma}{2\tau} \|x_1^k-x_2^k\| \leq \|\oFnor(z_1^k) -  \oFnor(z_2^k)\| \leq \alpha_k \|z_1^k-z_2^k\| \leq L\alpha_k \|x_1^k-x_2^k\| $
    %
    for all $k \geq K$. Due to $\alpha_k \to 0$, this yields a contradiction and hence, (\rnum{4}) is true.
\end{proof}

\begin{rem}[Nonlinear Programming]
    We now consider an application of \cref{thm4-1} to nonlinear programming $\nlp$, 
    \[ {\min}_{x\in \R^n} f(x)+\iota_X(x), \] 
    where $X:=\{x\in \R^n:~g(x)\leq 0,~h(x)=0\}$ and $g\in C^2(\R^n,\R^m)$, $h\in C^2(\R^n,\R^p)$. Let us set $F:=(g;h)$ and let $\bar x$ be a stationary point of $f+\iota_X$ satisfying $\nabla f(\bar x)+\D F(\bar x)^\top(\bar\lambda,\bar\mu)=0$ for some $\bar\lambda\in \R^m_+$, $\bar\mu\in\R^p$. Define the set $C :=\R^{m}_+\times \R^p$. Without loss of generality, we assume that all inequality constraints are active, i.e. $g(\bar x)=0$. We further assume that the strict Robinson's constraint qualification $(\mathrm{SRCQ})$ holds:
    \begin{equation}
        \label{CQ}
          \D F(\bar x)\R^n-N_C((\bar\lambda,\bar\mu))=\D F(\bar x)\R^n-N_{\R^m_+}(\bar\lambda)\times \{0_p\}=\R^{m+p}.
    \end{equation}
    Then, $\iota_X$ is $C^2$-strictly decomposable at $\bar x$ for $(\bar\lambda,\bar\mu)$ with decomposition pair $(\sigma_{C}, F)$. Therefore, \ref{B1} is satisfied. Next, we define the index set $\mathcal I :=\{i\in \{1,\dots,m\}: \bar\lambda_i>0\}$. Reordering the indices---if necessary---we may assume $\mathcal I =\{1,\dots,r\}$. The minimal exposed face $\cG$ of $C$ generated by $(\bar\lambda,\bar\mu)$ is $\R^r_+\times \{0_{m-r}\}\times \R^p$ and the corresponding normal cone satisfies $\cN = N_C((\bar\lambda,\bar\mu)) =\{0_r\}\times \R^{m-r}_-\times \{0_p\}$. Consider an arbitrary triple $(x,\lambda,\mu)$ sufficiently close to $(\bar x,\bar\lambda,\bar\mu)$ with $F(x)\in \ri(\cN)$ and $(\lambda,\mu)\in \ri(\cG)$. The condition $F(x)\in \ri(\cN)$ implies that $g_{i}(x)<0$ for $i > r$. 
    
    We now show that the projection $\proj$ onto $X$ is differentiable at $x+\tau \D F(x)^\top(\lambda, \mu)$ for some $\tau>0$. Our idea is to verify that $\iota_X$ is $C^2$-fully decomposable at $x$. Notice that, locally around $x$, $\iota_X$ agrees with the function $\iota_{Y}$, where $Y:=\{y\in \R^n: g_{i}(y)\leq 0,~h(y)=0, 1\leq i\leq r\}$. Therefore, it suffices to verify that $\iota_Y$ is $C^2$-fully decomposable at $x$. By \cite[Proposition 3.4 (ii)]{ouyang2023partI} and \eqref{CQ}, we may assume that:
    \begin{equation}
        \label{CQstrict}
        \D F(x)\R^n-N_{\R^m_+}(\bar\lambda)\times \{0_p\}=\D F(x)\R^n-N_{\R^m_+}(\lambda)\times \{0_p\}=\R^{m+p}.
    \end{equation}
    Let us set $\tilde F:=(g_{1};\dots;g_r;h)$. In particular, we have 
      \begin{equation}
        \label{CQstrict1}
        \D \tilde F(x)\R^n=\R^{r+p}.
    \end{equation}
    This proves that $\iota_Y$ (or $\iota_X$) is $C^2$-fully decomposable function at $x$ with decomposition pair $(\sigma_{\R^r_+\times \R^p},\tilde F)$. 
    Consequently, \ref{B3} holds for $\iota_X$ by \cref{remark_B3} and assumption \ref{B2} is satisfied since $\R^r_+\times \R^p$ is a polyhedral set. Moreover, applying \cite[Proposition 3.5 (ii)]{ouyang2023partI},  $\iota_X$ is also prox-regular at $\bar x$. As indicator functions are always prox-bounded, we can conclude that $\iota_X$ is globally prox-regular at $\bar x$ for any $v\in \partial \iota_X(\bar x)$. In particular, we know that \ref{A1} holds. Consequently, \cref{thm4-1} is applicable to $\nlp$ under the $(\mathrm{SRCQ})$. Let us also note that in the literature, the stronger linear independence constraint qualification $(\mathrm{LICQ})$ along with the strong second-order sufficient condition has been shown to be equivalent to the strong regularity of primal-dual solutions of certain KKT-type equations, see, e.g., \cite{bonnans2013perturbation,dontchev1996characterizations,robinson1980strongly,sun2006strong}. Our result here shows that for the normal map or the natural residual, such an equivalence can be established under the weaker condition $(\mathrm{SRCQ})$.
\end{rem}

\section{Conclusion}
\label{sec:conclu}
We study a strong second-order sufficient condition ($\SSOSC$) for composite-type problems whose nonsmooth part has a generalized conic quadratic second subderivative. The equivalence between the $\SSOSC$ and a second-order condition involving the normal map is established based on algebraic assumptions on the generalized Jacobian of the proximity operator. These structural conditions are then shown to hold for $C^2$-strictly decomposable functions under an additional mild geometric assumption on the support set $\pvpd$.
Finally, the equivalence between the $\SSOSC$ and other variational concepts are discussed, including CD-regularity and the strong metric regularity of the subdifferential, the normal map, and the natural residual. The developed theory is applicable to a broad and generic class of conic-type composite problems for which closed-form expressions of the corresponding proximity operator and generalized Jacobians are not necessarily known and can be utilized in the local convergence analysis of the semismooth Newton method. 

\revise{\section*{Acknowledgments} We would like to thank the Associate Editor and the anonymous reviewer for their detailed and constructive comments, which have helped greatly to improve the quality and presentation of the manuscript.}



\appendix

\section{Proof of \texorpdfstring{\cref{lem:exposed}}{Lemma 1.5}} \label{app:exposed}

\begin{proof} 
 To show (i), we note that $G_{\mathrm{ex}}(x) = \bigcap \{\partial\sigma_C(z) : x \in \partial\sigma_C(z)\}$. This automatically implies $N_{C}(x)\subseteq N_{C}(y)$. Let $v\in N_{C}(y)$, by \cite[Theorem 18.1]{rockafellar1970convex}, we know $G_{\mathrm{ex}}(x)\subset\partial\sigma_C(v)$, which proves $v\in N_C(x)$ and $N_C(x)=N_C(y)$. The equality for the tangent cone follows from $T_C(q)=N_{C}(q)^{\circ}$ for all $q\in C$. To show (ii), let us take $z\in\partial\sigma_C(y)$. By definition, we have $\langle z-x,y\rangle=\sigma_C(y)-\sigma_C(y)=0$ and $z-x\in T_C(x)=N_C(x)^{\circ}$. Consequently, it follows $z-x\in \{y\}^\perp\cap T_C(x)=N_{N_{C}(x)}(y)=\lin(T_C(x))$ by \cite[Example 6.40]{bauschke2017convex}. Then, for any $v\in N_C(x)$, we have $\langle v,x\rangle=\sigma_C(v)=\langle v,z\rangle$, which proves that $z\in \partial\sigma_C(v)$ and $z\in G_{\mathrm{ex}}(x)$. This yields $G_{\mathrm{ex}}(x)=\partial\sigma_C(y)$ since $G_{\mathrm{ex}}(x)\subseteq \partial\sigma_C(y)$ follows from the definition.    
\end{proof}

\section{\texorpdfstring{\cref{example:strict-decomp-not-full}}{Example 3.2}: Fully Decomposable Univariate Functions} \label{app:example}

\begin{prop} 
Suppose that $\vp:\R\to\R$ is $C^2$-fully decomposable at $0$. Then $\vp$ can be represented as $\max\{\vp_1,\vp_2\}$ locally around $0$ with $\vp_1,\vp_2$ being $C^2(\R)$.
\end{prop}
\begin{proof}
       \myrevise{Without loss of generality, we may assume $\vp(0)$=0.} Suppose $\vp=\sigma_C\circ F$ with $C\subseteq \R^m$ being nonempty, closed, and convex and $F:\R\to \R^m$ being $C^2$ on $\R$, and let the nondegeneracy condition $\D F(0)\R+\lin(N_C(\lambda)) =\R^m$ hold for some $\lambda \in C$. Due to $\dim(\D F(0)\R)\leq 1$, it then holds that $\dim(\lin(N_C(\lambda)))\geq m-1$ or, equivalently, $\dim(\aff(C))\leq 1$. Given the structure of 1-dimensional closed convex sets, there are only four different forms for $C$ (up to orthogonal transformations): a singleton, a half-line, a line, and a $1$-dimensional interval in $\R^m$. More specifically, we can consider the following cases. 

        \textbf{Case 1}: $C=\{\lambda\}$. In this case, we have $\vp=\langle \lambda, F \rangle$ locally around $0$.

        \textbf{Case 2}: $C=\lambda+\R e_1$, where $e_1=(1,0,\dots,0)^\top\in \R^m$. In this case, $\vp$ can be written as $\vp=\langle \lambda, F \rangle+ \iota_{\{0\}}(F_1)$ locally around $0$. Since $\vp$ is real-valued, it follows $\vp=\langle \lambda, F \rangle$ locally around $0$. 

        \textbf{Case 3}: $C=\lambda+\R_+ e_1$. In this case, we obtain $\vp=\langle \lambda, F \rangle+ \iota_{\R_-}(F_1)$ locally around $0$. Reusing the argument from \textbf{Case 2}, we can conclude that $\vp$ locally agrees with $\langle \lambda, F \rangle$ around $0$.

        \textbf{Case 4}: $C=\lambda+\R_- e_1$. This case is similar to \textbf{Case 3}.

        \textbf{Case 5}: $C=[\lambda,\lambda+de_1]$ for some $d\geq 0$. In this case, we can write $\vp=\langle \lambda, F \rangle+ \max\{ dF_1, 0\}$ locally around $0$.

       Since $\lambda \in C$ is arbitrary, we may choose $\lambda$ as a vertex of $C$, if it exists. Therefore, the discussed five cases cover all possible sets $C$, and in each case, we have $\vp = \max\{\vp_1, \vp_2\}$.
\end{proof}

\section{Proof of \texorpdfstring{\cref{prop3-7}}{Proposition 4.4}} \label{app:invert}

\begin{proof}
    Let $\tilde D$ be sufficiently close to $D$. According to \cite[Lemma 4.3]{sun2002strong}, there exist eigendecompositions of $\tilde D$ and $D$ such that:
    \[  \tilde D=\tilde P^\top \tilde \Lambda \tilde P, \quad D=P^\top \Lambda P, \quad \|P-\tilde P\|=\mathcal O(\|D-\tilde D\|),\quad\|\Lambda-\tilde \Lambda\|=\mathcal O(\|D-\tilde D\|).                     \]
    Here, $\Lambda$ and $\tilde \Lambda$ are given by $\Lambda=\diag(\lambda_1,\dots, \lambda_n)$ with $\lambda_1\geq \dots\geq \lambda_\ell > 0=\lambda_{\ell+1}=\dots=\lambda_n$ and $\tilde \Lambda=\diag(\tilde\lambda_1,\dots,\tilde\lambda_n)$ with $\tilde\lambda_1\geq \dots\geq \lambda_q>0=\lambda_{q+1}=\dots=\lambda_n$. By the Lipschitz continuity of the eigenvalues \cite[Theorem 4.3.1]{horn2012matrix}, we clearly have $q\geq \ell$ for all $\tilde D$ sufficiently close to $D$. Let us set $A=\diag(\lambda_1,\dots,\lambda_\ell)$, $\tilde A=\diag(\tilde \lambda_1,\dots,\tilde \lambda_\ell)$, and $A_q=\diag(\tilde \lambda_1,\dots,\tilde \lambda_q)$ and let $\Pi_r\in\R^{n\times r}$, $\Pi_r x := (x^\top,0)^\top$ be the natural projection onto the first $r$ components. The assumption is equivalent to
    \[  Q:=\Pi_\ell^\top PBP^\top\Pi_\ell+\tau^{-1}(A^{-1}-I)  \succeq \sigma I.           \]
   Now, let us take sufficiently small neighborhoods $U$ of $D$ and $V$ of $B$, such that for all $\tilde D\in U$ and $\tilde B\in V$, it holds that $\tilde Q := \Pi_l^\top \tilde P\tilde B\tilde P^\top\Pi_l+{\tau^{-1}}(\tilde A^{-1}-I)  \succeq \tfrac{\sigma}{2} I$.                
   
Let us note the choice of $P$ depends on $\tilde D$, but the errors $\|P-\tilde P\|$ and $\|A^{-1}-\tilde A^{-1}\|$ are uniform for all $\tilde D$ sufficiently close to $D$, so the error $\|Q-\tilde Q\|$ can be arbitrarily small provided that $\|P-\tilde P\|$, $\|A^{-1}-\tilde A^{-1}\|$, and $\|B-\tilde B\|$ are sufficiently small. The desired condition $\tilde D\tilde B\tilde D+\frac{1}{\tau}\tilde D(I-\tilde D)\succeq \frac{\sigma}{4}\tilde D^2$ can also be expressed as:
    \[ R := \Pi_q^\top \tilde P\tilde B\tilde P^\top\Pi_q+{\tau^{-1}}(A_q^{-1}-I)\succeq \tfrac{\sigma}{4}I.           \] 
    Next, let us introduce $L_1:=\spa\{e_1,\dots,e_\ell\}$ and $L_2:=\spa\{e_{\ell+1},\dots,e_q\}$. For every $h\in \R^q$, we can write $h=h_1+h_2$ with $h_1\in L_1$ and $h_2\in L_2$. It then follows:
    \begingroup
    \allowdisplaybreaks
    \begin{align*}
        h^\top Rh &=h_1^\top R h_1+2h_1^\top Rh_2+h_2^\top Rh_2
        \\&=h_1^\top \tilde Q h_1+2h_1^\top Rh_2+h_2^\top Rh_2\geq \tfrac{\sigma}{2}\|h_1\|^2-2|h_1^\top Rh_2|+h_2^\top Rh_2.
    \end{align*}
    \endgroup
    If $q=\ell$, then we are done since $h=h_1$. Now assume $q> \ell$. Then, we have:
    \[   h_2^\top Rh_2=   h_2^\top\Pi_q^\top \tilde P\tilde B\tilde P^\top\Pi_qh_2+{\tau^{-1}}h_2^\top(A_q^{-1}-I)h_2\geq[{\tau^{-1}}(\tilde\lambda_{\ell+1}^{-1}-1)-\|\tilde B\|]\|h_2\|^2,                  \]
    $|h_1^\top Rh_2|=|h_1^\top\Pi_q^\top \tilde P\tilde B\tilde P^\top\Pi_qh_2|\leq \|\tilde B\|\|h_1\|\|h_2\|$, and 
    by Young's inequality, we can infer $2\|\tilde B\|\|h_1\|\|h_2\|\leq \frac{\sigma}{4}\|h_1\|^2+\frac{4\|\tilde B\|^2}{\sigma} \|h_2\|^2$. 
    Therefore, we obtain
    \[ h^\top Rh\geq \tfrac{\sigma}{4}\|h_1\|^2+\left[{\tau^{-1}}(\tilde\lambda_{\ell+1}^{-1}-1)-\|\tilde B\|-\tfrac{4\|\tilde B\|^2}{\sigma}  \right]\|h_2\|^2.         \]
    Shrinking $U,V$ (if necessary), we can ensure ${\tau^{-1}}(\tilde\lambda_{\ell+1}^{-1}-1)-\|\tilde B\|-\frac{4\|\tilde B\|^2}{\sigma}\geq \frac{\sigma}{4}$, which finishes the proof.
\end{proof}

\vspace{-3ex}
\bibliographystyle{siamplain}
\bibliography{bib_part2}

@article {artacho2008characterization,
    	AUTHOR = {Arag\'{o}n Artacho, Francisco J. and Geoffroy, Michel H.},
     	TITLE = {Characterization of metric regularity of subdifferentials},
   	JOURNAL = {J. Convex Anal.},
  	FJOURNAL = {Journal of Convex Analysis},
    	VOLUME = {15},
      	YEAR = {2008},
    	NUMBER = {2},
     	PAGES = {365--380},
      	ISSN = {0944-6532,2363-6394},
   	MRCLASS = {49J53 (49J52 90C25)},
  	MRNUMBER = {2422996},
	MRREVIEWER = {Amos\ Uderzo},
}

@book {bauschke2017convex,
    	AUTHOR = {Bauschke, Heinz H. and Combettes, Patrick L.},
     	TITLE = {Convex analysis and monotone operator theory in {H}ilbert spaces},
    	SERIES = {CMS Books in Mathematics/Ouvrages de Math\'{e}matiques de la SMC},
   	EDITION = {Second},
 	PUBLISHER = {Springer, Cham},
      	YEAR = {2017},
     	PAGES = {xix+619},
      	ISBN = {978-3-319-48310-8; 978-3-319-48311-5},
   	MRCLASS = {49-02 (41A65 46B20 46C05 47H05 90C25)},
  	MRNUMBER = {3616647},
}

@article {benko2022second,
    AUTHOR = {Benko, Mat{\'u}{\v{s}} and Mehlitz, Patrick},
     TITLE = {Why second-order sufficient conditions are, in a way,
              easy---or---revisiting the calculus for second subderivatives},
   JOURNAL = {J. Convex Anal.},
  FJOURNAL = {Journal of Convex Analysis},
    VOLUME = {30},
      YEAR = {2023},
    NUMBER = {2},
     PAGES = {541--589},
}

@article {bonnans2005perturbation,
    	AUTHOR = {Bonnans, J. Fr\'{e}d\'{e}ric and Ram\'{\i}rez C., H\'{e}ctor},
     	TITLE = {Perturbation analysis of second-order cone programming problems},
   	JOURNAL = {Math. Program.},
  	FJOURNAL = {Mathematical Programming. A Publication of the Mathematical Programming Society},
    	VOLUME = {104},
      	YEAR = {2005},
    	NUMBER = {2-3},
     	PAGES = {205--227},
      	ISSN = {0025-5610,1436-4646},
   	MRCLASS = {90C31 (90C08)},
  	MRNUMBER = {2179235},
	MRREVIEWER = {Romesh\ Saigal},
}

@book {bonnans2013perturbation,
    	AUTHOR = {Bonnans, J. Fr\'{e}d\'{e}ric and Shapiro, Alexander},
     	TITLE = {Perturbation analysis of optimization problems},
    	SERIES = {Springer Series in Operations Research},
 	PUBLISHER = {Springer-Verlag, New York},
      	YEAR = {2000},
     	PAGES = {xviii+601},
      	ISBN = {0-387-98705-3},
   	MRCLASS = {90-02 (49-02 49K40 90C31)},
  	MRNUMBER = {1756264},
	MRREVIEWER = {T.\ Zolezzi},
}

@article {ByrChiNocOzt16,
    	AUTHOR = {Byrd, Richard H. and Chin, Gillian M. and Nocedal, Jorge and Oztoprak, Figen},
     	TITLE = {A family of second-order methods for convex {$\ell_1$}-regularized optimization},
   	JOURNAL = {Math. Program.},
  	FJOURNAL = {Mathematical Programming},
    	VOLUME = {159},
      	YEAR = {2016},
    	NUMBER = {1-2, Ser. A},
     	PAGES = {435--467},
      	ISSN = {0025-5610},
   	MRCLASS = {49M37 (65K05 90C30)},
  	MRNUMBER = {3535930},
	MRREVIEWER = {Lyudmila N. Polyakova},
}

@article {chan2008constraint,
    	AUTHOR = {Chan, Zi Xian and Sun, Defeng},
     	TITLE = {Constraint nondegeneracy, strong regularity, and nonsingularity in semidefinite programming},
   	JOURNAL = {SIAM J. Optim.},
  	FJOURNAL = {SIAM Journal on Optimization},
    	VOLUME = {19},
      	YEAR = {2008},
    	NUMBER = {1},
     	PAGES = {370--396},
      	ISSN = {1052-6234,1095-7189},
   	MRCLASS = {90C22 (65K05 65K10 90C25 90C31)},
  	MRNUMBER = {2403037},
	MRREVIEWER = {Qingzhi\ Yang},
}

@book {clarke1990optimization,
    	AUTHOR = {Clarke, F. H.},
     	TITLE = {Optimization and nonsmooth analysis},
    	SERIES = {Classics in Applied Mathematics},
    	VOLUME = {5},
   	EDITION = {Second},
 	PUBLISHER = {Society for Industrial and Applied Mathematics (SIAM), Philadelphia},
      	YEAR = {1990},
     	PAGES = {xii+308},
      	ISBN = {0-89871-256-4},
   	MRCLASS = {49-02 (01A75 49J52 58C20 90C48)},
  	MRNUMBER = {1058436},
}

@phdthesis {ding2012introduction,
  	AUTHOR = {Ding, Chao},
   	TITLE = {An introduction to a class of matrix optimization problems},
  	YEAR = {2012},
  	SCHOOL = {National University of Singapore}
}

@article {ding2014introduction,
    	AUTHOR = {Ding, Chao and Sun, Defeng and Toh, Kim-Chuan},
     	TITLE = {An introduction to a class of matrix cone programming},
   	JOURNAL = {Math. Program.},
  	FJOURNAL = {Mathematical Programming},
    	VOLUME = {144},
      	YEAR = {2014},
    	NUMBER = {1-2},
     	PAGES = {141--179},
      	ISSN = {0025-5610,1436-4646},
   	MRCLASS = {90C25},
  	MRNUMBER = {3179958},
	MRREVIEWER = {Baha\ Alzalg},
}

@article {dontchev1996characterizations,
    	AUTHOR = {Dontchev, A. L. and Rockafellar, R. T.},
     	TITLE = {Characterizations of strong regularity for variational inequalities over polyhedral convex sets},
   	JOURNAL = {SIAM J. Optim.},
  	FJOURNAL = {SIAM Journal on Optimization},
    	VOLUME = {6},
      	YEAR = {1996},
    	NUMBER = {4},
     	PAGES = {1087--1105},
      	ISSN = {1052-6234},
   	MRCLASS = {90C31 (49J40 49J52 49K40)},
  	MRNUMBER = {1416530},
	MRREVIEWER = {Adam\ B.\ Levy},
}

@book {DonRoc14,
    	AUTHOR = {Dontchev, Asen L. and Rockafellar, R. Tyrrell},
     	TITLE = {Implicit functions and solution mappings},
    	SERIES = {Springer Series in Operations Research and Financial Engineering},
   	EDITION = {Second},
 	PUBLISHER = {Springer, New York},
      	YEAR = {2014},
     	PAGES = {xxviii+466},
      	ISBN = {978-1-4939-1036-6; 978-1-4939-1037-3},
   	MRCLASS = {49J53 (26B10 47J30 47N10 49K40 58C15 90C31)},
  	MRNUMBER = {3288139},
	MRREVIEWER = {J.\ Borwein and Matthew\ K.\ Tam},
}

@article {DruLew13,
    	AUTHOR = {Drusvyatskiy, D. and Lewis, A. S.},
     	TITLE = {Tilt stability, uniform quadratic growth, and strong metric regularity of the subdifferential},
   	JOURNAL = {SIAM J. Optim.},
  	FJOURNAL = {SIAM Journal on Optimization},
    	VOLUME = {23},
      	YEAR = {2013},
    	NUMBER = {1},
     	PAGES = {256--267},
      	ISSN = {1052-6234,1095-7189},
   	MRCLASS = {49J53 (49J52 49K40 90C31)},
  	MRNUMBER = {3033107},
	MRREVIEWER = {Shawn\ Xianfu\ Wang},
}

@article {DanHarMal06,
    	AUTHOR = {Daniilidis, Aris and Hare, Warren and Malick, J\'{e}r\^{o}me},
     	TITLE = {Geometrical interpretation of the predictor-corrector type algorithms in structured optimization problems},
   	JOURNAL = {Optimization},
  	FJOURNAL = {Optimization. A Journal of Mathematical Programming and Operations Research},
    	VOLUME = {55},
      	YEAR = {2006},
    	NUMBER = {5-6},
     	PAGES = {481--503},
      	ISSN = {0233-1934,1029-4945},
   	MRCLASS = {90C30 (49J52 90C55)},
  	MRNUMBER = {2274921},
	MRREVIEWER = {Adam\ B.\ Levy},
}

@unpublished {HuTiaPanWen23,
	AUTHOR = {Hu, Jiang and Tian, Tonghua and Pan, Shaohua and Wen, Zaiwen},
  	TITLE = {On the local convergence of the semismooth Newton method for composite optimization},
  	NOTE = {arXiv preprint, arXiv:2211.01127v2},
  	YEAR = {2022}
}

@article {DruMorNhg14,
    	AUTHOR = {Drusvyatskiy, D. and Mordukhovich, B. S. and Nhgia, T. T. A.},
     	TITLE = {Second-order growth, tilt stability, and metric regularity of the subdifferential},
   	JOURNAL = {J. Convex Anal.},
  	FJOURNAL = {Journal of Convex Analysis},
    	VOLUME = {21},
      	YEAR = {2014},
    	NUMBER = {4},
     	PAGES = {1165--1192},
      	ISSN = {0944-6532},
   	MRCLASS = {49J52 (49J53 49K40 90C31)},
  	MRNUMBER = {3331214},
	MRREVIEWER = {Francisco J. Arag\'{o}n Artacho},
}

@book {facchinei2003finite,
    	AUTHOR = {Facchinei, Francisco and Pang, Jong-Shi},
     	TITLE = {Finite-dimensional variational inequalities and complementarity problems. {V}ol. {II}},
    	SERIES = {Springer Series in Operations Research},
 	PUBLISHER = {Springer-Verlag, New York},
      	YEAR = {2003},
     	PAGES = {i-xxxiv, 625--1234 and II1--II57},
      	ISBN = {0-387-95581-X},
   	MRCLASS = {90-02 (15A39 49J40 65K10 90C33)},
  	MRNUMBER = {1955649},
	MRREVIEWER = {Daniel\ Ralph},
}

@article {fukushima1981proximal,
    	AUTHOR = {Fukushima, Masao and Mine, Hisashi},
     	TITLE = {A generalized proximal point algorithm for certain nonconvex minimization problems},
   	JOURNAL = {Internat. J. Systems Sci.},
  	FJOURNAL = {International Journal of Systems Science},
    	VOLUME = {12},
      	YEAR = {1981},
    	NUMBER = {8},
     	PAGES = {989--1000},
      	ISSN = {0020-7721},
     	CODEN = {IJSYA9},
   	MRCLASS = {90C30 (65K05)},
  	MRNUMBER = {628084 (83b:90137)},
	MRREVIEWER = {L. Grippo},
}

@article {gfrerer2022scd,
	AUTHOR = {Gfrerer, Helmut and Mandlmayr, Michael and Outrata, Ji{\v{r}}{\'\i} V and Valdman, Jan},
  	TITLE = {On the SCD semismooth* Newton method for generalized equations with application to a class of static contact problems with Coulomb friction},
  	JOURNAL = {Comput. Optim. Appl.},
  	FJOURNAL = {Computational Optimization and Applications},
 	YEAR = {2022},
	ISSN = {1573-2894},
}

@article {guo2015some,
    	AUTHOR = {Guo, Shao-Yan and Zhang, Li-Wei and Hao, Shou-Lin},
     	TITLE = {On some aspects of perturbation analysis for matrix cone optimization induced by spectral norm},
   	JOURNAL = {J. Oper. Res. Soc. China},
  	FJOURNAL = {Journal of the Operations Research Society of China},
    	VOLUME = {3},
      	YEAR = {2015},
    	NUMBER = {3},
     	PAGES = {275--296},
      	ISSN = {2194-668X,2194-6698},
   	MRCLASS = {90C22 (90C30 90C46)},
  	MRNUMBER = {3394472},
	MRREVIEWER = {Qinghong\ Zhang},
}

@article {hale2008fixed,
    	AUTHOR = {Hale, Elaine T. and Yin, Wotao and Zhang, Yin},
     	TITLE = {Fixed-point continuation for {$l_1$}-minimization: methodology and convergence},
   	JOURNAL = {SIAM J. Optim.},
  	FJOURNAL = {SIAM Journal on Optimization},
    	VOLUME = {19},
     	YEAR = {2008},
    	NUMBER = {3},
     	PAGES = {1107--1130},
      	ISSN = {1052-6234,1095-7189},
   	MRCLASS = {90C06 (65K10 90C25 90C90)},
  	MRNUMBER = {2460734},
	MRREVIEWER = {C.\ Ilioi},
}

@book {HirLem01,
    	AUTHOR = {Hiriart-Urruty, Jean-Baptiste and Lemar\'{e}chal, Claude},
     	TITLE = {Fundamentals of convex analysis},
    	SERIES = {Grundlehren Text Editions},
 	PUBLISHER = {Springer-Verlag, Berlin},
      	YEAR = {2001},
     	PAGES = {x+259},
      	ISBN = {3-540-42205-6},
   	MRCLASS = {90-01 (26B25 49-01)},
  	MRNUMBER = {1865628},
}

@book {horn2012matrix,
    	AUTHOR = {Horn, Roger A. and Johnson, Charles R.},
     	TITLE = {Matrix analysis},
   	EDITION = {Second},
 	PUBLISHER = {Cambridge University Press, Cambridge},
      	YEAR = {2013},
     	PAGES = {xviii+643},
      	ISBN = {978-0-521-54823-6},
   	MRCLASS = {15-01},
  	MRNUMBER = {2978290},
	MRREVIEWER = {Mohammad\ Sal\ Moslehian},
}

@article {kato2007sqp,
    	AUTHOR = {Kato, Hirokazu and Fukushima, Masao},
     	TITLE = {An {SQP}-type algorithm for nonlinear second-order cone programs},
   	JOURNAL = {Optim. Lett.},
  	FJOURNAL = {Optimization Letters},
    	VOLUME = {1},
      	YEAR = {2007},
    	NUMBER = {2},
     	PAGES = {129--144},
      	ISSN = {1862-4472,1862-4480},
   	MRCLASS = {90C55 (90C22)},
  	MRNUMBER = {2357594},
}

@article {kenderov1975semi,
    	AUTHOR = {Kenderov, P.},
     	TITLE = {Semi-continuity of set-valued monotone mappings},
   	JOURNAL = {Fund. Math.},
  	FJOURNAL = {Polska Akademia Nauk. Fundamenta Mathematicae},
    	VOLUME = {88},
      	YEAR = {1975},
    	NUMBER = {1},
     	PAGES = {61--69},
      	ISSN = {0016-2736,1730-6329},
   	MRCLASS = {54C60 (46G05)},
  	MRNUMBER = {380723},
	MRREVIEWER = {S.\ T.\ Lin},
}

@incollection {kummer1992newton,
    	AUTHOR = {Kummer, Bernd},
     	TITLE = {Newton's method based on generalized derivatives for nonsmooth functions: convergence analysis},
 	BOOKTITLE = {Advances in {O}ptimization ({L}ambrecht, 1991)},
    	SERIES = {Lecture Notes in Econom. and Math. Systems},
    	VOLUME = {382},
     	PAGES = {171--194},
 	PUBLISHER = {Springer, Berlin},
      	YEAR = {1992},
      	ISBN = {3-540-55446-7},
   	MRCLASS = {90C30 (49J52 65H05)},
  	MRNUMBER = {1229731},
	MRREVIEWER = {Lionel\ Thibault},
}

@incollection {kummer1988newton,
    	AUTHOR = {Kummer, Bernd},
     	TITLE = {Newton's method for nondifferentiable functions},
 	BOOKTITLE = {Advances in {M}athematical {O}ptimization},
    	SERIES = {Math. Res.},
    	VOLUME = {45},
     	PAGES = {114--125},
 	PUBLISHER = {Akademie-Verlag, Berlin},
      	YEAR = {1988},
      	ISBN = {3-05-500543-0},
   	MRCLASS = {90C30 (49D15 65H05 90C48)},
  	MRNUMBER = {953328},
	MRREVIEWER = {Robert\ W.\ Owens},
}

@article{kunisch2016time,
  title={Time optimal control for a reaction diffusion system arising in cardiac electrophysiology--a monolithic approach},
  author={Kunisch, Karl and Pieper, Konstantin and Rund, Armin},
  fjournal={ESAIM: Mathematical Modelling and Numerical Analysis},
  journal = {ESAIM Math. Model. Numer. Anal.},
  volume={50},
  number={2},
  pages={381--414},
  year={2016},
  publisher={EDP Sciences},
}

@article {li2018efficient,
    AUTHOR = {Li, Xudong and Sun, Defeng and Toh, Kim-Chuan},
     TITLE = {A highly efficient semismooth {N}ewton augmented {L}agrangian
              method for solving lasso problems},
   JOURNAL = {SIAM J. Optim.},
  FJOURNAL = {SIAM Journal on Optimization},
    VOLUME = {28},
      YEAR = {2018},
    NUMBER = {1},
     PAGES = {433--458},
}

@article {lin2019clustered,
    AUTHOR = {Lin, Meixia and Liu, Yong-Jin and Sun, Defeng and Toh,
              Kim-Chuan},
     TITLE = {Efficient sparse semismooth {N}ewton methods for the clustered
              {L}asso problem},
   JOURNAL = {SIAM J. Optim.},
  FJOURNAL = {SIAM Journal on Optimization},
    VOLUME = {29},
      YEAR = {2019},
    NUMBER = {3},
     PAGES = {2026--2052},
}

@article {liu2019characterizations,
    	AUTHOR = {Liu, Yong-Jin and Li, Ruonan and Wang, Bo},
     	TITLE = {On the characterizations of solutions to perturbed {$l_1$} conic optimization problem},
   	JOURNAL = {Optimization},
  	FJOURNAL = {Optimization. A Journal of Mathematical Programming and Operations Research},
    	VOLUME = {68},
      	YEAR = {2019},
    	NUMBER = {6},
     	PAGES = {1157--1186},
      	ISSN = {0233-1934,1029-4945},
   	MRCLASS = {90C31 (49J53 90C26 90C46)},
  	MRNUMBER = {3960847},
	MRREVIEWER = {Le Thanh Tung},
}

@article {mannel2021hybrid,
    AUTHOR = {Mannel, Florian and Rund, Armin},
     TITLE = {A hybrid semismooth quasi-{N}ewton method for nonsmooth
              optimal control with {PDE}s},
   JOURNAL = {Optim. Eng.},
  FJOURNAL = {Optimization and Engineering. International Multidisciplinary
              Journal to Promote Optimization Theory \& Applications in
              Engineering Sciences},
    VOLUME = {22},
      YEAR = {2021},
    NUMBER = {4},
     PAGES = {2087--2125},
}

@phdthesis{milzarek2016numerical,
	AUTHOR = {Milzarek, Andre},
  	TITLE = {Numerical methods and second order theory for nonsmooth problems},
  	YEAR = {2016},
  	SCHOOL = {Technische Universit{\"a}t M{\"u}nchen}
}

@article {milzarek2014semismooth,
    	AUTHOR = {Milzarek, Andre and Ulbrich, Michael},
     	TITLE = {A semismooth {N}ewton method with multidimensional filter globalization for {$l_1$}-optimization},
   	JOURNAL = {SIAM J. Optim.},
  	FJOURNAL = {SIAM Journal on Optimization},
    	VOLUME = {24},
      	YEAR = {2014},
    	NUMBER = {1},
     	PAGES = {298--333},
      	ISSN = {1052-6234,1095-7189},
   	MRCLASS = {65K10 (49M15 90C53)},
  	MRNUMBER = {3166973},
	MRREVIEWER = {Giles\ Auchmuty},
}

@article {milzarek2019stochastic,
    AUTHOR = {Milzarek, Andre and Xiao, Xiantao and Cen, Shicong and Wen,
              Zaiwen and Ulbrich, Michael},
     TITLE = {A stochastic semismooth {N}ewton method for nonsmooth
              nonconvex optimization},
   JOURNAL = {SIAM J. Optim.},
  FJOURNAL = {SIAM Journal on Optimization},
    VOLUME = {29},
      YEAR = {2019},
    NUMBER = {4},
     PAGES = {2916--2948},
}

@article {mordukhovich2015full,
    	AUTHOR = {Mordukhovich, B. S. and Nghia, T. T. A. and Rockafellar, R. T.},
     	TITLE = {Full stability in finite-dimensional optimization},
   	JOURNAL = {Math. Oper. Res.},
  	FJOURNAL = {Mathematics of Operations Research},
    	VOLUME = {40},
      	YEAR = {2015},
    	NUMBER = {1},
     	PAGES = {226--252},
      	ISSN = {0364-765X,1526-5471},
   	MRCLASS = {90C31 (49J53 90C22)},
  	MRNUMBER = {3320421},
	MRREVIEWER = {Phan Qu\cfac{o}c Kh\'{a}nh},
}

@article {mordukhovich2012second,
    	AUTHOR = {Mordukhovich, B. S. and Rockafellar, R. T.},
     	TITLE = {Second-order subdifferential calculus with applications to tilt stability in optimization},
   	JOURNAL = {SIAM J. Optim.},
  	FJOURNAL = {SIAM Journal on Optimization},
    	VOLUME = {22},
      	YEAR = {2012},
    	NUMBER = {3},
     	PAGES = {953--986},
      	ISSN = {1052-6234,1095-7189},
   	MRCLASS = {49J52 (26E25 90C31)},
  	MRNUMBER = {3023759},
	MRREVIEWER = {Andrew\ C.\ Eberhard},
}

@unpublished {ouyang2021trust,
  title={{A trust region-type normal map-based semismooth Newton method for nonsmooth nonconvex composite optimization}},
  author={Ouyang, Wenqing and Milzarek, Andre},
  note={preprint arXiv:2106.09340},
  year={2021}
}

@unpublished {ouyang2023partI,
	AUTHOR = {Ouyang, Wenqing and Milzarek, Andre},
  	TITLE = {Variational properties of decomposable functions. Part I: Strict epi-calculus and applications},
  	NOTE = {preprint arXiv:2311.07267},
  	YEAR = {2023}
}

@phdthesis{pieper2015finite,
  title={Finite element discretization and efficient numerical solution of elliptic and parabolic sparse control problems},
  author={Pieper, Konstantin},
  year={2015},
  school={Technische Universit{\"a}t M{\"u}nchen}
}

@article {poliquin1996generalized,
    	AUTHOR = {Poliquin, R. A. and Rockafellar, R. T.},
     	TITLE = {Generalized {H}essian properties of regularized nonsmooth functions},
   	JOURNAL = {SIAM J. Optim.},
  	FJOURNAL = {SIAM Journal on Optimization},
    	VOLUME = {6},
      	YEAR = {1996},
    	NUMBER = {4},
     	PAGES = {1121--1137},
      	ISSN = {1052-6234},
     	CODEN = {SJOPE8},
   	MRCLASS = {49J52 (26B05 58C20)},
  	MRNUMBER = {1416532 (97j:49025)},
}

@article {PolRoc98,
    	AUTHOR = {Poliquin, R. A. and Rockafellar, R. T.},
     	TITLE = {Tilt stability of a local minimum},
   	JOURNAL = {SIAM J. Optim.},
  	FJOURNAL = {SIAM Journal on Optimization},
    	VOLUME = {8},
      	YEAR = {1998},
    	NUMBER = {2},
     	PAGES = {287--299},
      	ISSN = {1052-6234,1095-7189},
   	MRCLASS = {90C31 (49J52 49K40)},
  	MRNUMBER = {1618790},
	MRREVIEWER = {B.\ Mordukhovich},
}

@article {qi1993convergence,
    	AUTHOR = {Qi, Li Qun},
     	TITLE = {Convergence analysis of some algorithms for solving nonsmooth equations},
   	JOURNAL = {Math. Oper. Res.},
  	FJOURNAL = {Mathematics of Operations Research},
    	VOLUME = {18},
      	YEAR = {1993},
    	NUMBER = {1},
     	PAGES = {227--244},
      	ISSN = {0364-765X,1526-5471},
   	MRCLASS = {65H10 (49J52 90C30)},
  	MRNUMBER = {1250115},
}

@article {qi1993nonsmooth,
    	AUTHOR = {Qi, Li Qun and Sun, Jie},
     	TITLE = {A nonsmooth version of {N}ewton's method},
   	JOURNAL = {Math. Program.},
  	FJOURNAL = {Mathematical Programming},
    	VOLUME = {58},
      	YEAR = {1993},
    	NUMBER = {3},
     	PAGES = {353--367},
      	ISSN = {0025-5610,1436-4646},
   	MRCLASS = {90C30 (49J52 65H05)},
  	MRNUMBER = {1216791},
	MRREVIEWER = {Aharon\ Ben-Tal},
}

@article {robinson1980strongly,
    	AUTHOR = {Robinson, Stephen M.},
     	TITLE = {Strongly regular generalized equations},
   	JOURNAL = {Math. Oper. Res.},
  	FJOURNAL = {Mathematics of Operations Research},
    	VOLUME = {5},
      	YEAR = {1980},
    	NUMBER = {1},
     	PAGES = {43--62},
      	ISSN = {0364-765X,1526-5471},
   	MRCLASS = {90C48 (58C15 90C31)},
  	MRNUMBER = {561153},
	MRREVIEWER = {M.\ B.\ Suryanarayana},
}

@article {robinson1992normal,
    AUTHOR = {Robinson, Stephen M.},
     TITLE = {Normal maps induced by linear transformations},
   JOURNAL = {Math. Oper. Res.},
  FJOURNAL = {Mathematics of Operations Research},
    VOLUME = {17},
      YEAR = {1992},
    NUMBER = {3},
     PAGES = {691--714},
}

@book {rockafellar1970convex,
    	AUTHOR = {Rockafellar, R. Tyrrell},
     	TITLE = {Convex analysis},
    	SERIES = {Princeton Mathematical Series},
    	VOLUME = {No. 28},
 	PUBLISHER = {Princeton University Press, Princeton},
      	YEAR = {1970},
     	PAGES = {xviii+451},
   	MRCLASS = {26.52 (46.00)},
  	MRNUMBER = {274683},
	MRREVIEWER = {Ky\ Fan},
}

@article {rockafellar2023augmented,
    	AUTHOR = {Rockafellar, R. Tyrrell},
     	TITLE = {Augmented {L}agrangians and hidden convexity in sufficient conditions for local optimality},
   	JOURNAL = {Math. Program.},
  	FJOURNAL = {Mathematical Programming},
    	VOLUME = {198},
      	YEAR = {2023},
    	NUMBER = {1},
     	PAGES = {159--194},
      	ISSN = {0025-5610,1436-4646},
   	MRCLASS = {90C46 (49J53 90C30)},
  	MRNUMBER = {4550949},
}

@book {rockafellar2009variational,
    	AUTHOR = {Rockafellar, R. Tyrrell and Wets, Roger J.-B.},
     	TITLE = {Variational analysis},
    	SERIES = {Grundlehren der mathematischen Wissenschaften},
    	VOLUME = {317},
 	PUBLISHER = {Springer-Verlag, Berlin},
      	YEAR = {1998},
     	PAGES = {xiv+733},
      	ISBN = {3-540-62772-3},
   	MRCLASS = {49-02 (46N10 47N10 49J52 49K40 90C30)},
  	MRNUMBER = {1491362},
	MRREVIEWER = {Francis\ H.\ Clarke},
}

@book {schneider2014convex,
    	AUTHOR = {Schneider, Rolf},
     	TITLE = {Convex bodies: the {B}runn-{M}inkowski theory},
    	SERIES = {Encyclopedia of Mathematics and its Applications},
    	VOLUME = {151},
 	PUBLISHER = {Cambridge University Press, Cambridge},
      	YEAR = {2014},
     	PAGES = {xxii+736},
      	ISBN = {978-1-107-60101-7},
   	MRCLASS = {52-02 (52A20 52A39)},
  	MRNUMBER = {3155183},
	MRREVIEWER = {Andrea\ Colesanti},
}

@article {shapiro2003class,
    	AUTHOR = {Shapiro, Alexander},
     	TITLE = {On a class of nonsmooth composite functions},
   	JOURNAL = {Math. Oper. Res.},
  	FJOURNAL = {Mathematics of Operations Research},
    	VOLUME = {28},
      	YEAR = {2003},
    	NUMBER = {4},
     	PAGES = {677--692},
      	ISSN = {0364-765X,1526-5471},
   	MRCLASS = {90C31 (90C46)},
  	MRNUMBER = {2015908},
	MRREVIEWER = {Adam\ B.\ Levy},
}

@article {stadler2009elliptic,
    	AUTHOR = {Stadler, Georg},
     	TITLE = {Elliptic optimal control problems with {$L^1$}-control cost and applications for the placement of control devices},
   	JOURNAL = {Comput. Optim. Appl.},
  	FJOURNAL = {Computational Optimization and Applications. An International Journal},
    	VOLUME = {44},
      	YEAR = {2009},
    	NUMBER = {2},
     	PAGES = {159--181},
      	ISSN = {0926-6003,1573-2894},
   	MRCLASS = {49K20 (49M37)},
  	MRNUMBER = {2556849},
	MRREVIEWER = {Boris\ Vexler},
}

@article {stella2017forward,
    	AUTHOR = {Stella, Lorenzo and Themelis, Andreas and Patrinos, Panagiotis},
     	TITLE = {Forward-backward quasi-{N}ewton methods for nonsmooth optimization problems},
   	JOURNAL = {Comput. Optim. Appl.},
  	FJOURNAL = {Computational Optimization and Applications. An International Journal},
    	VOLUME = {67},
     	YEAR = {2017},
    	NUMBER = {3},
     	PAGES = {443--487},
      	ISSN = {0926-6003,1573-2894},
   	MRCLASS = {90C53 (49M15 65K10)},
  	MRNUMBER = {3654182},
	MRREVIEWER = {Jinhai\ Chen},
}

@article {sun2006strong,
    	AUTHOR = {Sun, Defeng},
     	TITLE = {The strong second-order sufficient condition and constraint nondegeneracy in nonlinear semidefinite programming and their implications},
   	JOURNAL = {Math. Oper. Res.},
  	FJOURNAL = {Mathematics of Operations Research},
    	VOLUME = {31},
      	YEAR = {2006},
    	NUMBER = {4},
     	PAGES = {761--776},
      	ISSN = {0364-765X,1526-5471},
   	MRCLASS = {90C22 (90C31)},
  	MRNUMBER = {2281228},
	MRREVIEWER = {Ji-Ming\ Peng},
}

@article {sun2002strong,
    	AUTHOR = {Sun, Defeng and Sun, Jie},
     	TITLE = {Strong semismoothness of eigenvalues of symmetric matrices and its application to inverse eigenvalue problems},
   	JOURNAL = {SIAM J. Numer. Anal.},
  	FJOURNAL = {SIAM Journal on Numerical Analysis},
    	VOLUME = {40},
      	YEAR = {2002},
    	NUMBER = {6},
     	PAGES = {2352--2367},
      	ISSN = {0036-1429,1095-7170},
   	MRCLASS = {15A18 (65F18 65H17)},
  	MRNUMBER = {1974190},
	MRREVIEWER = {Fabio\ Di Benedetto},
}

@book {ulbrich2011semismooth,
    AUTHOR = {Ulbrich, Michael},
     TITLE = {Semismooth {N}ewton methods for variational inequalities and
              constrained optimization problems in function spaces},
    SERIES = {MOS-SIAM Series on Optimization},
    VOLUME = {11},
 PUBLISHER = {Society for Industrial and Applied Mathematics (SIAM); Mathematical Optimization Society,
              Philadelphia, PA},
      YEAR = {2011},
     PAGES = {xiv+308},
}

@unpublished {wangstrong,
 	AUTHOR = {Wang, Shiwei and Ding, Chao and Zhang, Yangjing and Zhao, Xinyuan},
  	TITLE = {Strong variational sufficiency for nonlinear semidefinite programming and its implications},
	NOTE = {preprint arXiv:2210.04448},
  	YEAR = {2022}
}

@article {wang2009properties,
    	AUTHOR = {Wang, Yun and Zhang, Liwei},
     	TITLE = {Properties of equation reformulation of the {K}arush-{K}uhn-{T}ucker condition for nonlinear second order cone optimization problems},
   	JOURNAL = {Math. Methods Oper. Res.},
  	FJOURNAL = {Mathematical Methods of Operations Research},
    	VOLUME = {70},
      	YEAR = {2009},
    	NUMBER = {2},
     	PAGES = {195--218},
      	ISSN = {1432-2994,1432-5217},
   	MRCLASS = {90C46 (49J53 90C22 90C31 90C55)},
  	MRNUMBER = {2558410},
	MRREVIEWER = {Yves\ Lucet},
}

@article {wang2010nonsingularity,
    	AUTHOR = {Wang, Yun and Zhang, LiWei},
     	TITLE = {Nonsingularity in second-order cone programming via the smoothing metric projector},
   	JOURNAL = {Sci. China Math.},
  	FJOURNAL = {Science China. Mathematics},
    	VOLUME = {53},
      	YEAR = {2010},
    	NUMBER = {4},
     	PAGES = {1025--1038},
      	ISSN = {1674-7283,1869-1862},
   	MRCLASS = {90C46 (65K05 90C30)},
  	MRNUMBER = {2640186},
	MRREVIEWER = {Olga\ Brezhneva},
}

@unpublished {weis2010note,
	AUTHOR = {Weis, Stephan},
  	TITLE = {A note on touching cones and faces},
  	NOTE = {preprint arXiv:1010.2991},
  	YEAR = {2010}
}

@article {wen2012convergence,
    	AUTHOR = {Wen, Zaiwen and Yin, Wotao and Zhang, Hongchao and Goldfarb, Donald},
     	TITLE = {On the convergence of an active-set method for {$\ell_1$} minimization},
   	JOURNAL = {Optim. Methods Softw.},
  	FJOURNAL = {Optimization Methods \& Software},
    	VOLUME = {27},
      	YEAR = {2012},
    	NUMBER = {6},
     	PAGES = {1127--1146},
      	ISSN = {1055-6788,1029-4937},
   	MRCLASS = {90C25 (65K05 90C06)},
  	MRNUMBER = {2955298},
	MRREVIEWER = {\'{E}milie\ Chouzenoux},
}

@article {zhao2010newtoncg,
    AUTHOR = {Zhao, Xin-Yuan and Sun, Defeng and Toh, Kim-Chuan},
     TITLE = {A {N}ewton-{CG} augmented {L}agrangian method for semidefinite
              programming},
   JOURNAL = {SIAM J. Optim.},
  FJOURNAL = {SIAM Journal on Optimization},
    VOLUME = {20},
      YEAR = {2010},
    NUMBER = {4},
     PAGES = {1737--1765},
}

\end{document}